\documentclass[10pt,a4paper]{amsart}
\usepackage{amsfonts}
\usepackage{amsthm}
\usepackage{amsmath}
\usepackage{amscd}
\usepackage[latin2]{inputenc}
\usepackage{t1enc}
\usepackage[mathscr]{eucal}
\usepackage{indentfirst}
\usepackage{graphicx}
\usepackage{graphics}
\usepackage{hyperref}
\numberwithin{equation}{section}
\usepackage[margin=2.9cm]{geometry}
\usepackage{epstopdf} 
\usepackage[all]{xy}
\usepackage[usenames, dvipsnames]{color}
\usepackage{amssymb}

\newcommand{\A}{\mathcal{A}}
\newcommand{\B}{\mathcal{B}}

\newcommand{\Ff}{\mathcal{F}}
\newcommand{\Spec}{\text{Spec}}
\newcommand{\Hom}{\text{Hom}}

\newcommand{\Fib}{\text{Fib}}

\newcommand{\Z}{\mathbb{Z}}

\newcommand{\OO}{\mathcal{O}}
\newcommand{\p}{\mathfrak{p}}
\newcommand{\q}{\mathfrak{q}}

\newcommand{\M}{\mathcal{M}}
\newcommand{\U}{\mathcal{U}}
\newcommand{\CC}{\mathcal{C}}

\newcommand\restr[2]{{
  \left.\kern-\nulldelimiterspace 
  #1 
  \vphantom{\big|} 
  \right|_{#2} 
  }}

\graphicspath{ {images/} }
\newtheorem{theorem}{Theorem}[section]
\newtheorem*{theorem*}{Theorem}

\theoremstyle{definition}
\newtheorem{definition}[theorem]{Definition}
\newtheorem*{definition*}{Definition}
\newtheorem{lemma}[theorem]{Lemma}
\newtheorem*{lemma*}{Lemma}
\newtheorem{proposition}[theorem]{Proposition}
\newtheorem*{proposition*}{Proposition}
\newtheorem{corollary}{Corollary}[theorem]
\newtheorem{example}{Example}[section]
\newtheorem*{corollary*}{Corollary}
\theoremstyle{remark}
\newtheorem{claim}{Claim}
\theoremstyle{remark}
\newtheorem{remark}{\textit{Remark}}[section]

\begin{document}

\title{\'Etale Covers and Fundamental Groups of schematic Finite Spaces}

\author{J. S\'anchez Gonz\'alez, C. Tejero Prieto}
\address{Departamento de Matem\'{a}ticas, Instituto de F\'isica Fundamental y Matem\'aticas, Universidad de Salamanca,
Plaza de la Merced 1-4, 37008 Salamanca, Spain}
\email{javier14sg@usal.es}
\email{carlost@usal.es}


\thanks {The authors were supported by research project MTM2017-86042-P (MEC). The first author was supported by Santander and Universidad de Salamanca.The second author was supported by the GIR STAMGAD (JCyL)}

\maketitle

\begin{abstract} 

We introduce the category of finite \'etale covers of an arbitrary schematic finite space $X$ and show that, equipped with an appropriate natural fiber functor, it is a Galois Category. This allows us to define the \'etale fundamental group of schematic spaces. If $X$ is a finite model of a scheme $S$, we show that the resulting Galois theory on $X$ coincides with the classical theory of finite \'etale covers on $S$ and therefore we recover the classical \'etale fundamental group introduced by Grothendieck. In order to prove these results it is crucial to find a suitable geometric notion of connectedness for schematic finite spaces and also to study  their geometric points. We achieve these goals by means of the strong cohomological constraints enjoyed by schematic finite spaces.
\end{abstract}

\section*{Introduction}

Let $S$ be a quasi-compact and quasi-separated scheme (qc-qs) and consider a finite affine covering $\mathcal{U}=\{U_i\}$ such that for every $s\in S$, the open set $U^s:=\bigcap_{s\in U_i}U_i$ is also affine. A classical topological construction that goes back to McCord \cite{McCord} allows us to define a finite poset (equivalently, a $T_0$ finite topological space) $X$ associated to $(S, \mathcal{U})$, along with a natural continuous projection $\pi:S\longrightarrow X$. 

In his paper \cite{Fernando schemes}, F. Sancho showed that if we endow $X$ with the sheaf of rings $\OO_X$ defined by the push forward of the structure sheaf $\OO_S$, then $(X, \OO_X)$ is a ringed space such that the projection $\pi$ induces an  equivalence between the categories of quasi-coherent sheaves $\mathbf{Qcoh}(X)\simeq \mathbf{Qcoh}(S)$ and an isomorphism at the level of their cohomology. We will call these finite ringed topological spaces \textit{finite models} of our schemes. Furthermore, one can recover the original scheme from any finite model via a \textit{recollement} procedure.  We can also interpret these spaces as representations of a quiver with values in the category of commutative rings with unit, or simply as a way of ordering descent data. 

In general, ringed finite topological spaces enjoy a number of nice properties that do not hold for arbitrary ringed or locally ringed spaces and provide a nice framework to study schemes from their <<building blocks>> (see \cite{Fernando schemes},\cite{preprint cohaces}), providing a more constructive and tangible approach to scheme theory (and more general locally ringed spaces) that still allows us to do what one would intuitively call <<geometry>>. For instance, a sheaf is defined by giving a finite set of data and the appropriate restriction morphisms (corresponding to the partial order of the underlying poset), the conditions of quasi-coherence and  coherence adopt a simple combinatorial form, in the topological case we have a well-behaved notion of \textit{cosheaf} that allows us to define homology for arbitrary sheaves (because the derived category has enough $K$-projectives in this finite case), every sheaf has explicit finite acyclic resolutions, etc.

In order to formalize these ideas, F. Sancho (\cite{Fernando schemes}) identified an adequate class of ringed finite topological spaces whose objects share many <<geometric>> properties with schemes and called them \textit{schematic finite spaces}. With an appropriate notion of \textit{schematic morphisms} that behaves well with respect to quasi-coherent sheaves, they define a category: $\mathbf{SchFin}$. This category contains all finite models of schemes and is strictly larger than it, as shown in Example \ref{example:schematic-not-scheme}. Furthermore, he defined a localizing family of morphisms in this category, whose elements are called \textit{qc-isomorphisms} (because they preserve the categories of quasi-coherent sheaves) or sometimes \textit{weak equivalences} (in the language of homotopy categories); such that for the localized category $\mathbf{SchFin}_{qc}$ there is a fully faithful functor
\begin{equation*}
\Phi:\textbf{Schemes}^{qc-qs}\longrightarrow \mathbf{SchFin}_{qc}
\end{equation*}
whose image can be described explicitly. Moreover, via a generalized \textit{recollement} procedure,  $\mathbf{SchFin}_{qc}$ gets identified with a subcategory $\mathbf{PSchemes}$ of the category \textbf{LRS} of locally ringed spaces that behaves like ${\textbf{Schemes}}^{qc-qs}$ at least at the cohomological level. The category $\mathbf{PSchemes}$ is larger than $\mathbf{Schemes}$, but in a different way than algebraic spaces. Aspects like the derived geometry of schemes can also be studied using adequate finite models, recovering advanced classical results like Grothendieck's general duality theorem (\cite{Fernando y JF}).

The category $\mathbf{PSchemes}$ is actually the subcategory of locally ringed spaces obtained by gluing affine schemes in the \textit{site of flat monomorphisms} of schemes. The basic theory and good behavior of these morphisms in problems of descent were originally considered by Raynaud in \cite{Raynaud}. Additionally, this site is closely related to the site of pro-open immersions (and coincides with it in some cases, see \cite{prufer spaces}; see also the analogous relationship between weak \'etale morphisms and pro-\'etale morphisms in \cite{scholze}), which has been recently employed in the contexts of birrational geometry and adic spaces to study Pr\"ufer algebraic spaces and Pr\"ufer pairs (see \cite{adic tame} and \cite{prufer spaces}).

Regarding the combinatorial side of our perspective, although in some ways it may seem similar to other constructive approaches such as simplicial schemes, it will become apparent that our objects are much more rigid and better behaved with respect to natural geometric notions. \medskip

From a purely topological perspective, the construction $\pi:S\longrightarrow X$ also provides results concerning homotopical notions, for the time being just in a traditional sense (see \cite{Fernando homotopy} for a ringed version of Stone's homotopical classification). In particular, one can show that if the covering $\mathcal{U}:=\{U_i\}$ defining $\pi$ is chosen appropriately, then there is an induced isomorphism at the level of fundamental groups
\begin{align*}
\pi_*:\pi_1(S, s)\overset{\sim}{\longrightarrow}\pi_1(X, \pi(s)).
\end{align*}
In other words, the categories of locally constant sheaves on both $S$ and $X$ are equivalent (actually, they have the same weak homotopy type, as $X$ plays the role of the \textit{\v{C}ech nerve of $\mathcal{U}$} in this context). Fundamental groups of (finite) posets can be described explicitly in terms of \textit{edge-paths} (see \cite{edgepaths}) and used to approach problems of group cohomology, since one can construct a (finite) poset whose fundamental group is any given (finite) group.

\medskip

This topological analogy and the good behavior of cohomology theories in $\mathbf{SchFin}$ are our main motivations to study the problem of existence of an analogue of Grothendieck's \textit{\'etale} fundamental group $\pi_1^{et}$ for schematic finite spaces using the standard axiomatization of Galois Categories. To achieve this, we introduce beforehand a number of new techniques regarding schematic finite spaces.

The first key point is finding a notion of connectedness in $\mathbf{SchFin}$ that is geometric. Not only in the sense that it reflects the ordinary connectedness of the locally ringed space associated to a given schematic finite space $X$ (connectedness of the spectrum of $\OO_X(X)$ already does this and in general is stronger than connectedness of the underlying poset), but also in that it allows us to perform geometric constructions such as the decomposition into <<connected>> components. This notion is called \textit{well-connectedness}. In order to define it, we strengthen connectedness of $\Spec(\OO_X(X))$ with an \textit{a priori} unrelated property called \textit{pw-connectedness},  that requires all stalk rings of $X$ to have connected spectrum. This depends on how the algebraic data contained in $X$ is organized. Eventually, we prove the following result, which showcases the geometric behavior of well-connectedness and justifies its importance.

\begin{theorem*}[\ref{theorem connectedness}]
A schematic finite space $X$ is well-connected if and only if for every decomposition $X=X_1\amalg X_2$ in $\mathbf{SchFin}$, either $X_1$ or $X_2$ is qc-isomorphic to $\emptyset$.
\end{theorem*}

The proof relies on the fact that pw-connected spaces being well-behaved with respect to the notion of connectedness, in the sense that for them, well-connectedness and ordinary topological connectedness are equivalent. 

Moreover, pw-connected spaces define a subcategory $\mathbf{SchFin}^{pw}\subseteq \mathbf{SchFin}$ and we prove that this inclusion has a right adjoint $\mathbf{pw}$ (Theorem \ref{theorem pw connectification}) such that the natural counit $\mathbf{pw}(X)\to X$ is a qc-isomorphism for every $X$. Since the categories we will be considering are stable under qc-isomorphisms (Section \ref{section stability}), this will allow us to assume pw-connectedness without loss of generality.

Before introducing \'etale covers, we spend some time describing the \textit{geometric points} of $X$, understood as elements of $X^\bullet(\Omega)$ for $\Omega$ an algebraically closed field. If $X$ is a finite model of a scheme $S$,  we show that hey are in bijection with the geometric points of $S$. The discussion makes apparent that the condition of being schematic, despite its \textit{a priori} cohomological definition, is far from arbitrary and from only being relevant in cohomology theory. This fact critically distinguishes our spaces and morphisms from more abstract forms of <<ordered descent data>>.

Finally, we define the main ingredients of this paper. For this, we  notice that the category of finite \'etale covers of a scheme $S$, $\textbf{Fet}_S$, understood as sheaves of quasi-coherent algebras, is equivalent to a certain subcategory of sheaves on any of its finite models $X$, denoted $\mathbf{Qcoh}^{fet}(X)\subset \mathbf{Qcoh}(X)$. This category can be defined even if $X$ is a schematic finite space that is not a finite model of a scheme. Furthermore, we endow it with a natural fiber functor $\mathrm{Fib}_{\overline{x}}$  (for some $\overline{x}\in X^\bullet(\Omega)$). This fiber functor, which \textit{a priori} is defined using scheme-theoretic tools, will coincide with (the underlying set of) a fibered product in $\mathbf{SchFin}^{pw}$ (rather than in $\mathbf{SchFin}$), thus retaining the original geometric meaning of the concept of <<fibers>>.

Moreover, $\mathbf{Qcoh}(X)$ is the underlying category of a \textit{site} $X_{\mathbf{Qcoh}}^{fppf}$ in such a way that $\mathbf{Qcoh}^{fet}(X)$ is simply its corresponding category of finite locally constant objects, just as in the case of the fppf or \'etale topologies for schemes. After all these reductions within the same qc-isomorphism class, we prove the main result:

\begin{theorem*}[\ref{theorem main}]
Let $X$ be a schematic finite space whose ring of global sections has connected spectrum and let $\overline{x}\in X^\bullet(\Omega)$ be a geometric point. Then $(\mathbf{Qcoh}^{fet}(X), \mathrm{Fib}_{\overline{x}})$ is a Galois Category (Definition \ref{definition galois category}). Furthermore, if $X$ is the finite model of a scheme $S$ and $\overline{s}\in S^\bullet(\Omega)$ is the corresponding geometric point (\ref{proposition points scheme are schematic points of the model}), then there is an isomorphism of profinite groups
\begin{equation*}
\pi_1^{et}(S, \overline{s})\simeq \pi_1^{et}(X, \overline{x}),
\end{equation*}
where $\pi_1^{et}(X, \overline{x}):=\mathrm{Aut}_{[\mathbf{Qcoh}^{fet}(X)^{\mathrm{op}}, \mathbf{Set}_f]}(\mathrm{Fib}_{\overline{x}})$. 
\end{theorem*}

Besides the previous citations, some work has been done regarding \'etale fundamental groups and descent data, for example the proof of a general version of the Seifert-Van Kampen Theorem in \cite{van kampen}. In future work we will recover this type of results using the very general machinery of sheaves and stacks on finite spaces and  use them to provide computational examples. We also hope to be able to adapt the theory of \'Etale Homotopy to our context, as well as to extend the pro-\'etale topology (\cite{scholze}) to the realm of finite schematic spaces, since they are compatible with the definition of weakly \'etale ring homomorphisms without additional considerations.  Actually, most of our theory should generalize, \textit{mutatis mutandis}, to arbitrary ringed posets (or even more generally, to Kolmogorov spaces); where the non-finite ones correspond to non-quasi-compact spaces.
\medskip

Let us describe now the structure of the paper. Section 1 reviews many basic properties of schematic finite spaces. In Section 2 we develop the main techniques for this paper, defining the subcategory of pw-connected finite schematic spaces $\mathbf{SchFin}^{pw}$, and proving Theorem \ref{theorem connectedness}. We also describe geometric points of schematic finite spaces and their corresponding fiber functors. Section 3 introduces and describes the notion of finite \'etale covers in this context, proving that they can be understood as locally constant objects in a precise sense. In Section 4 we prove the invariance of the pair $(\mathbf{Qcoh}^{fet}(X), \mathrm{Fib}_{\overline{x}})$ under qc-isomorphisms (which applied in combination with section 2, reduces our problem to the better behaved subcategory $\mathbf{SchFin}^{pw}$). In Section 5 we prove our main result (Theorem \ref{theorem main}). In Section 6 we outline some basic examples, that will be further elaborated in future works, and give some ideas about the potential usefulness of \'etale fundamental groups of schematic spaces for studying singularities. 

In this paper all ringed spaces, unless otherwise explicitly mentioned, are assumed to be stalkwise Noetherian, see Definition \ref{defi:stalkwise-Noetherian}.

\section{Generalities on Schematic Finite Spaces}

In this section, for the reader's convenience and in order to fix notations, we compile a series of general results and definitions concerning ringed finite topological spaces as developed in \cite{Fernando schemes} and \cite{Fernando homotopy}, with a few additional results (Corollaries \ref{corollary qcoh restrictions between affines}, \ref{corollary caracterizacion afines} and Proposition \ref{prop cohomological characterization qciso}) and---hopefully---insightful remarks.
\medskip

Let $X$ be a finite $T_0$ topological space, or equivalently, a finite poset. Given a point $x\in X$, we denote by $U_x$ the minimal open set containing $x$. In terms of the order defined by the topology of $X$ one has $U_x=\{x'\in X:x\leq x'\}$. Similarly, $C_x:=\overline{\{x\}}=\{x'\in X:x'\leq x\}$ is the closure of $x$.

A sheaf of abelian groups (or sets, rings, etc.) $\mathcal{F}\in\textbf{Sh}(X)$ is equivalent to giving, for every $x\in X$, an abelian group $\mathcal{F}_x$ and, for every $x\leq y$, a restriction morphism $f_{xy}:\mathcal{F}_x\rightarrow \mathcal{F}_y$ such that for every $x\leq y\leq z$, $f_{xz}=f_{yz}\circ f_{xy}$. Therefore, the category of abelian sheaves on a finite topological space $X$ coincides with the representations of its associated Hasse quiver with values in the category of abelian groups. This holds because the minimal open sets $\{U_x\}_{x\in X}$ are a basis for the topology and the sheaf condition allows us to compute sections on any other open subset. However, this is clearly not true for presheaves. The stalk of $\mathcal{F}$ at $x$ is just $\mathcal{F}_x= \mathcal{F}(U_x)$. 
A morphism of sheaves $\phi:\mathcal{F}\rightarrow \mathcal{G}$ is determined by a collection of morphisms $\phi_x:\mathcal{F}_x\rightarrow \mathcal{G}_x$ (for all $x\in X$) that are compatible with restriction maps, i.e. it is a morphism of representations of the Hasse quiver of $X$.

\begin{remark}[Data on $X$]\label{remark covariant sheaves}
Let us introduce some very general terminology. Recall that the topological dual of $X$, denoted $X^*$, is constructed by reversing the partial order of $X$. If we see $X$ as a category (whose objects are the points and such that there is an arrow $x\to y$ whenever $x\leq y$), $X^*$ is just the opposite category $X^{\mathrm{op}}$. 

In general, given any category $\mathfrak{C}$, we can define the \textit{category of $\mathfrak{C}$-data on $X$} as the category of functors $\mathfrak{C}\text{-}\mathbf{data}_X:=\mathbf{Func}(X, \mathfrak{C})$. It is equivalent to the category of sheaves with values on $\mathfrak{C}$ if $\mathfrak{C}$ has finite limits---a datum $\Ff$ defines a sheaf $\Ff_{\mathrm{sh}}\colon X_{\mathrm{top}}^{\mathrm{op}}\to\mathfrak{C}$ via $\Ff_{\mathrm{sh}}(U)=\varprojlim_{x\in U}\mathfrak{C}(x)$, with $X_{\mathrm{top}}$ the topological site of $X$---and to the category of \textit{cosheaves} on $X$ if $\mathfrak{C}$ has finite colimits (see \cite{preprint cohaces} for a study of the homology  of cosheaves of abelian groups on finite ringed spaces). We may also define the category of \textit{$\mathfrak{C}$-codata on $X$} as $\mathfrak{C}\text{-}\mathbf{codata}_X:=\mathfrak{C}\text{-}\mathbf{data}_{X^*}=\mathfrak{C}^{\mathrm{op}}\text{-}\mathbf{data}_{X}$.

We can allow the base poset to vary as well: if $f\colon X\to Y$ is a continuous map, we have a functor $f^*\colon\mathfrak{C}\text{-}\mathbf{data}_Y\to \mathfrak{C}\text{-}\mathbf{data}_X$. Now define $\mathfrak{C}\text{-}\mathbf{data}$ to be the category of $\mathfrak{C}$-data on any poset, whose objects are pairs $(X, \Ff)$. Define a morphism $(X, \Ff)\to (Y, \mathcal{G})$ to be a pair $f\colon X\to Y$ and a natural transformation $f^\#\colon f^*\mathcal{G}\to\Ff$. These are an example of \textit{heteromorphisms} and  the replacement of maps between simplicial objects in the language of data on posets. Note that, if $\mathbf{CRing}$ denotes the category of commutative rings with unit, $\mathbf{CRing}\text{-}\mathbf{data}$ is the category of finite ringed spaces. Similarly, one defines $\mathfrak C\text{-} \mathbf{codata}:=\mathfrak C^{\mathrm{op}}\text{-}\mathbf{data}$.

We can also change the category of ``coefficients'': if $\Phi\colon\mathfrak{C}\to\mathfrak{C'}$ is a functor, we have an induced functor $\Phi_*\colon \mathfrak{C}\text{-}\mathbf{data}\to \mathfrak{C}'\text{-}\mathbf{data}$ that sends every $\Ff$ to $\Phi\circ \Ff$.

This idea admits a generalization if we consider $2$-categories and $2$-limits (or $2$-colimits). For instance, if $\mathfrak{C}=\mathbf{Cat}$ is the $2$-category of categories, the category of $\mathbf{Cat}$-data is precisely the category of stacks (understood merely as $2$-sheaves of categories) on the topological site $X_{\mathrm{top}}$.
\end{remark}

In a (finite) poset, every sheaf $\mathcal{F}$ has a natural (finite) acyclic resolution $C^\bullet\mathcal{F}$ called the \textit{standard resolution}. Its sections on $U\subseteq X$ are given by:
\begin{align}\label{equation standard resolution}
(C^i\mathcal{F})(U):=\prod_{x_i>...>x_0\in U}\mathcal{F}_{x_i} & & i\geq 0
\end{align}
with the obvious restriction morphisms and a natural differential given by alternate sums (see \cite{Fernando y JF}).

\medskip

Now, let $(X, \OO_X)$ be a finite \textit{ringed} space and denote by $r_{xx'}:\OO_{X, x}\longrightarrow\OO_{X, x'}$ the restriction morphisms of its structure sheaf. 

\begin{definition}\label{defi:stalkwise-Noetherian}
We say that a ringed space $(X,\mathcal O_X)$ is stalkwise Noetherian if $\mathcal O_{X,x}$ is a Noetherian ring for every point $x\in X$.
\end{definition}

We have the following characterization of quasi-coherent ($\mathbf{Qcoh}(X)$) and coherent sheaves on a finite ringed space $X$:
\begin{proposition}\label{proposition: characterization qcoh, ftype, coh}
Let $(X, \OO_X)$ be a finite ringed space and $\M$ a sheaf of $\OO_X$-modules.
\begin{itemize}
\item[1)]$\M\in\mathbf{Qcoh}(X)$ if and only if for every $x\leq x'$, $\M_x\otimes_{\OO_{X, x}}\OO_{X, x'}\simeq \M_{x'}$.
\item[2)]$\M$ is of finite type if and only if for every $x\in X$, $\M_x$ is a finite $\OO_{X, x}$-module and for every $x\leq x'$, $\M_x\otimes_{\OO_{X, x}}\OO_{X, x'}\longrightarrow \M_{x'}$ is surjective. 
\item[3)]$\M$ is coherent if and only if it is of finite type and verifies that <<for every $x\in X$, every sub-$\OO_{X, x}$-module of finite type $\mathcal{N}\subset \M_x$ is of finite presentation and $\mathcal{N}\otimes_{\OO_{X, x}}\OO_{X, x'}\rightarrow \M_{x'}$ is injective>>.
\end{itemize}
\end{proposition}
\begin{corollary}\label{corollary qcoh is abelian}
$\OO_X$ is coherent if and only if for every $x\in X$, every finitely-generated ideal of $\OO_{X, x}$ is of finite presentation and, for every $x\leq x'$, $r_{xx'}$ is flat. 

In particular, if $X$ is stalkwise Noetherian, then $\OO_X$ is coherent if and only if $r_{xx'}$ is flat for every $x\leq x'$. Furthermore, in this case $\mathbf{Qcoh}(X)$ is a Grothendieck abelian subcategory of the category of all sheaves of modules.
\end{corollary}
In the rest of the paper all ringed spaces are assumed to be stalkwise Noetherian. If more general objects were necessary, techniques of Noetherian reduction may be applied. Under this hypothesis, Corollary \ref{corollary qcoh is abelian} motivates the following definition.

\begin{definition}\label{definition finite spaces}
A \textit{finite space} $X$ is a ringed finite topological space whose restriction morphisms $r_{xx'}:\OO_{X, x}\rightarrow \OO_{X, x'}$ are flat.
\end{definition}

\begin{remark} The category of topological spaces $\textbf{Top}$ can be seen as a fully faithful subcategory of $\textbf{RS}$ (ringed spaces) via $X\mapsto (X, \Z_X)$. Quasi-coherent sheaves of $\Z-$modules are exactly locally constant sheaves of abelian groups, which enables an easy description of covers and the fundamental group, as seen in \cite{Fernando homotopy}.
\end{remark}
\subsection{Finite models and \textit{recollement}}
Certain finite ringed spaces serve as models for quasi-compact schemes in the sense that they encode ordered descent data. 

In general, let $S$ be a quasi-compact topological space. Consider a finite open covering $\mathcal{U}=\{U_i\}$ and for every $s\in S$ define $U^s:=\bigcap_{s\in U_i}U_i$. We can consider the following equivalence relation on $S$:
\begin{align*}
s\sim t\Longleftrightarrow U^s=U^t
\end{align*}
and endow the finite set $X:=S/\sim$ with the preorder given by $
[s]\leq[t]\Longleftrightarrow U^t\subseteq U^s$.

The finite topological space $X$ (with the natural continuous projection $\pi:S\longrightarrow X$) is called \textit{the finite model} of the pair $(S, \mathcal{U})$ or simply \textit{a} finite model of $S$.

\begin{remark}
Actually, the finite model of $(S, \U)$ is the prime spectrum of the finite distributive lattice generated by $\U$, denoted $\tau_\U$, as a sublattice of the topology of $S$, denoted $\tau_S$. The projection is the natural map $S\to\Spec(\tau_S)\to\Spec(\tau_\U)=X$.
\end{remark}

Let $S$ be a quasi-compact and quasi-separated (qc-qs) scheme and consider a finite affine covering $\mathcal{U}=\{U_i\}$ such that $U^s$ is affine for every $s$. Such coverings exist thanks to the quasi-separatedness hypothesis and will be called \textit{locally affine coverings}. 

Let $\pi:S\longrightarrow X$ be the finite model associated to the topological space $S$ and the covering $\mathcal{U}$. Now, we endow $X$ with the sheaf of rings $\OO_X=\pi_*\OO_S$, which verifies $\OO_{X, x}=\OO_S(U^s)$ for any $s$ such that $\pi(s)=x$. The ringed finite space $X\equiv (X, \OO_X)$ (with the natural projection $\pi\colon S\to X$) is called \textit{the finite model} associated to the pair $(S, \mathcal{U})$ or simply a finite model of the scheme $S$. Notice that $X$ is a finite space (Definition \ref{definition finite spaces}).

The choice of locally affine coverings ensures that the notion of quasi-coherent sheaf is <<trivial>> at each $\pi^{-1}(U_x)=U^s$. This implies the following fundamental result:

\begin{proposition}[\cite{Fernando homotopy}, Theorem 2.9]
Let $\pi:S\longrightarrow X$ be as above, then the morphism of topoi $(\pi^*, \pi_*):\mathbf{Sh}(S)\to \mathbf{Sh}(X)$ restricts to an equivalence $$(\pi^*, \pi_*):\mathbf{Qcoh}(S)\simeq \mathbf{Qcoh}(X)$$ such that $\mathrm{R}^i\pi_*\M=0$ for every $\M\in \mathbf{Qcoh}(S)$ and $i>0$.
\end{proposition}

Conversely, given a ringed finite topological space $X$, we have a \textit{recollement} functor, denoted $\mathfrak{Spec}$, that constructs a locally ringed space. It is defined as follows:
\begin{itemize}
\item For $x\leq x'$, the restriction morphism induces $r_{xx'}^*\colon\Spec(\OO_{X, x'})\longrightarrow\Spec(\OO_{X, x})$. 
\item We obtain a codatum $\Spec(\OO_X)\in\mathrm{Ob}(\mathbf{LRS}\text{-}\mathbf{codata}_X)$ and define:
$$\mathfrak{Spec}(X):=\mathrm{colim}\,\Spec(\OO_{X}).$$
\end{itemize}
This colimit always exists, coinciding with the colimit as ringed spaces ($\mathbf{RS}$) because the forgetful functor $\mathbf{LRS}\to\mathbf{RS}$ preserves all colimits, and satisfies the following universal property: for every ringed space $Y$ there is a bijection
\begin{align*}
\Hom_{\textbf{LRS}}(
\mathrm{colim}\,\Spec(\OO_{X}), Y)\simeq \Hom_{\mathbf{LRS}\text{-}\mathbf{codata}_X}(\Spec(\OO_X), Y).
\end{align*}
In other words $\mathfrak{Spec}$ is a left-adjoint to the natural fully faithful inclusion $\mathbf{LRS}\hookrightarrow \mathbf{LRS}\text{-}\mathbf{codata}_X$ sending every locally ringed space to the constant functor. Explicitly, it is a quotient of $\coprod_x\Spec(\OO_{X, x})$ by a suitable equivalence relation determined by the poset. A morphism of ringed finite spaces $f\colon X\to Y$ induces $\mathfrak{Spec}(f)\colon\mathfrak{Spec}(X)\to \mathfrak{Spec}(Y)$ in $\mathbf{LRS}$.

\medskip

Of course, we can also write this colimit as the coequalizer (and equalizer at the level of sheaves of rings) of a suitable collection of arrows. Notice that $\OO_{\mathfrak{Spec}(X)}(\mathfrak{Spec}(X))=\OO_X(X)$. Furthermore, if every morphism of the inductive system is an open immersion of schemes, then $\mathfrak{Spec}(X)$ is a scheme (by standard \textit{recollement}). In particular, if $X$ is a finite model of a scheme $S$, we have $\mathfrak{Spec}(X)\simeq S$. In this case, we will say that $X$ has \textit{open restrictions}. A classical characterization due to Grothendieck says that this happens if and only if the restriction maps are flat epimorphisms of rings that are of finite presentation (see \cite[IV, 17.9,1]{EGAIV}).

\begin{remark}
If $\mathbf{CRings}^{\mathrm{op}}\simeq\mathbf{AffSchemes}\subseteq \mathbf{LRS}$ denotes the category of affine schemes as a subcategory of locally ringed spaces, we have (Remark \ref{remark covariant sheaves}) that $\mathfrak{Spec}$ coincides with the composition:
\begin{align*}
\mathbf{CRings}\text{-}\mathbf{data}\simeq \mathbf{AffSchemes}\text{-}\mathbf{codata}\to \mathbf{LRS}\text{-}\mathbf{codata}\overset{\mathrm{colim}}{\to} \mathbf{LRS}.
\end{align*}
\end{remark}

\subsection{Affine and Schematic finite spaces. Qc-isomorphisms.}

In this section we identify an appropriate subcategory of finite spaces (Definition \ref{definition finite spaces}) that captures the local behavior of schemes, as introduced in \cite{Fernando schemes}. This also allows us to describe ringed spaces more general than schemes. The omitted proofs can be found in \cite{Fernando schemes}. For a detailed analysis and a combinatorial description of these spaces  we refer the reader to  \cite{fernandoypedro}.

\begin{definition}
Let $X$ be a finite space. We say that $X$ is \textit{affine} if
\begin{itemize}
\item[1)]$X$ is acyclic, i.e. $H^i(X, \OO_X)=$ for $i>0$.
\item[2)]The projection to a point $\pi:X\longrightarrow (\star, \OO_X(X))$ induces an equivalence of quasi-coherent topoi $(\pi^*, \pi_*)\colon \mathbf{Qcoh}(X)\simeq \OO_X(X)\textbf{-mod}$ (notice that $\pi_*=\Gamma(X, -)$). 
\end{itemize}
\end{definition}
In particular, thanks to the characterization of quasi-coherent modules, every minimal open set $U_x$ is easily seen to be affine. Note that the second condition says that, for every $x\in X$ and $\M\in\mathbf{Qcoh}(X)$, there is a natural <<localization>> isomorphism
\begin{align}\label{equation fiber of qcoh in affine spaces}
\M(X)\otimes_{\OO_X(X)}\OO_{X, x}\overset{\sim}{\to}\M_x.
\end{align}
\begin{proposition}
Let $X$ be a finite space. The following are equivalent:
\begin{itemize}
\item $X$ is affine.
\item $X$ is acyclic and every $\M\in\mathbf{Qcoh}(X)$ is generated by its global sections (which means that $\M(X)\otimes_{\OO_{X}(X)}\OO_{X, x}\to\M_x$ is surjective for every $x\in X$).
\item Every $\M\in\mathbf{Qcoh}(X)$ is generated by its global sections and is acyclic. 
\end{itemize}
\end{proposition}
An important property that will be of use later is the following:
\begin{proposition}\label{proposition characterization of affines}
Let $X$ be an affine finite space, then $\OO_X(X)\longrightarrow\prod_{x\in X}\OO_{X, x}$ is faithfully flat. In particular, for any $U\subset X$ affine, $\OO_X(X)\longrightarrow\OO_X(U)$ is flat.
\end{proposition}

\begin{corollary}\label{corollary qcoh restrictions between affines}
If $X$ is finite, $x\in X$, $U\subseteq U_x$ is an affine open subset and $\M\in\mathbf{Qcoh}(X)$, there is an isomorphism $\M_x\otimes_{\OO_{X, x}}\OO_X(U)\simeq \M(U)$.
\begin{proof}
Consider the natural morphism $\M_x\otimes_{\OO_{X, x}}\OO_X(U)\to\M(U)$. Tensoring by the faithfully flat morphism $\OO_X(U)\to\prod_{y\in U}\OO_{X, y}$ of the previous proposition we obtain a morphism
\begin{align*}
\M_x\otimes_{\OO_{X, x}}\prod_{y\in U}\OO_{X, y}\to\M(U)\otimes_{\OO_X(U)}\prod_{y\in U}\OO_{X, y};
\end{align*}
but since $\M\in\mathbf{Qcoh}(X)$ and $U$ is affine, it is isomorphic to the identity $\prod_{y\in U}\M_y\to \prod_{y\in U}\M_y$. By faithful flatness, we conclude.
\end{proof}
\end{corollary}

\begin{definition}
Let $f:X\longrightarrow Y$ be a morphism (of ringed spaces) between finite spaces. We say that $f$ is an affine morphism if it preserves quasi-coherence (i.e. for any $\M\in\mathbf{Qcoh}(X)$, $\mathbb{R}f_*\M\in {D}_{qc}(Y)$) and $f^{-1}(U_y)$ is affine for every $y\in Y$. 
\end{definition}

From now on, given a morphism $f:X\longrightarrow Y$, $ \Gamma_f:X\longrightarrow X\times Y$ will denote its graphic. If $f=Id_X$, $\Gamma_f\equiv \Delta$ is the diagonal of $X$.
\begin{definition}[Schematic finite spaces and morphisms]\label{definition schematic}
We say that a morphism of ringed spaces $f:X\longrightarrow Y$ is \textit{schematic} if $\mathbb{R}\Gamma_{f*}\OO_X\in D_{qc}(X\times Y)$. Equivalently, if for every $i\geq 0$, $x\leq x'\in X$, $y\leq y'\in Y$, we have isomorphisms
\begin{align}
& H^i(U_x\cap f^{-1}(U_y), \OO_X)\otimes_{\OO_{X, x}}\OO_{X, x'}\overset{\sim}{\longrightarrow}H^i(U_{x'}\cap f^{-1}(U_y), \OO_X),\\
& H^i(U_x\cap f^{-1}(U_y), \OO_X)\otimes_{\OO_{X, y}}\OO_{X, y'}\overset{\sim}{\longrightarrow}H^i(U_{x}\cap f^{-1}(U_{y'}), \OO_X).\nonumber
\end{align}
We say that a space $X$ is \textit{schematic} if $Id_X:X\longrightarrow X$ is schematic. 
\end{definition}

The condition of $f$ being schematic is local on $X$, and if $Y$ is schematic, it is also local on $Y$. In particular, a space $X$ is schematic if and only if $U_x$ is schematic for all $x\in X$. Additionally, one proves that for any $\M\in\mathbf{Qcoh}(X)$, $f$ schematic implies that $\mathbb{R}f_*\M\in D_{qc}(Y)$ itself if quasi-coherent; and the converse holds if $X$ itself is schematic.

For $i=0$, Definition \ref{definition schematic} yields the first part of the following fundamental property:
\begin{proposition}\label{prop restrictions are flat epic}
If $X$ is schematic, the restrictions $r_{xx'}:\OO_{X, x}\longrightarrow\OO_{X, x'}$ are flat epimorphisms in the category of rings. Additionally, if $X$ is affine, $r_x\colon \OO_X(X)\to \OO_{X, x}$ is also a flat ring epimorphism.
\end{proposition}

\begin{remark}
Flat epimorphisms of rings $A\to B$, which <<only>> differ from open immersions of affine schemes in that we remove the condition finite presentation, are quite interesting objects. They are in correspondence (modulo isomorphism) with certain \textit{smashing} subcategories of the derived category $D(A)$, as well as with \textit{coherent} subsets of $\Spec(A)$ closed under generalization. Furthermore, they are related to the notion of \textit{universal localizations} (obtaining a bijection under hypothesis such as $\Spec(A)$ being a smooth algebraic variety). If $A$ is Noetherian and $A\to B$ is an epimorphism (so $B$ is Noetherian), it is flat if and only $\mathrm{Tor}_1^A(B, B)=0$. These ideas are developed in \cite{lidia} and \cite{krauser}.
\end{remark}
The following is a classical result from commutative algebra that sheds light into the local structure of these morphisms and will be enough for the treatment here:
\begin{proposition}\cite[Prop. 2.4.]{lazard}\label{proposition flat epic are local iso}
Let $\phi:A\longrightarrow B$ be a morphism of commutative rings with unit. $\phi$ is a flat epimorphism if and only if for every $\p\in\mathrm{Spec}(B)$, the localized map $A_{\phi^{-1}(\p)}\overset{\sim}{\to} B_\p$ is an isomorphism. Furthermore, this holds if and only if for every $\q\in \mathrm{Spec}(A)$, either $B=\q B$ or $A_\q\overset{\sim}{\to} B_\q$ is an isomorphism. In other words, $i:\mathrm{Spec}(B)\hookrightarrow \mathrm{Spec}(A)$ is a flat monomorphism of affine schemes if and only if $i^{-1}\widetilde{B}\simeq \widetilde{A}$, where, for a ring $R$, $\widetilde{R}$ denotes the structure sheaf $\mathrm{Spec}(R)$.
\end{proposition}

If $X$ is a schematic finite space, then $\mathfrak{Spec}(X)$ is a gluing of affine schemes along flat monomorphisms of affine schemes. 
\begin{proposition}[Extension Theorem]\label{theorem extension of quasicoherent sheaves}
Let $X$ be a schematic space and $j:U\hookrightarrow X$ the inclusion of an open subset; then $j$ is schematic. In particular the extension
\begin{align*}
j_*:\mathbf{Qcoh}(U)\longrightarrow\mathbf{Qcoh}(X)
\end{align*}
is fully faithful.
\end{proposition}

\begin{corollary}\label{corollary caracterizacion afines}
Let $X$ be a schematic finite space and let $x\in X$ be a point. An open subset $U\subseteq U_x$ is affine if and only if $\OO_X(U)\to \prod_{x\in U}\OO_{X, x'}$ is faithfully flat.
\begin{proof}
One direction is just Proposition \ref{proposition characterization of affines}. Assume that $\OO_X(U)\to \prod_{x\in U}\OO_{U, x'}$ is faithfully flat. To see that any quasi-coherent module is generated by its global sections, apply the extension theorem to $U\hookrightarrow U_x$ and note that $U_x$ is affine because $X$ is schematic. Furthermore $H^i(U, \OO_U)=0$ for every $i>0$ due to the faithful flatness hypothesis, which follows from an easy computation using the standard resolution of the structure sheaf. We conclude by the characterization of affine finite spaces.
\end{proof}
\end{corollary}

We will denote by $\mathbf{SchFin}$ the category whose objects are schematic finite spaces and whose morphisms are schematic morphisms. From now on we will always assume that we are working in this category, unless explicitly mentioned otherwise.

\begin{theorem}
$\mathbf{SchFin}$ has finite fibered products. 
\end{theorem}

Given $f:X\longrightarrow Z$, $g:Y\longrightarrow Z$, their fibered product $X\times_ZY$ is the topological fibered product equipped with the sheaf of rings such that for every $(x, y)\in X\times_ZY$ with image $z$, $\OO_{X\times_ZY, (x, y)}=\OO_{X, x}\otimes_{\OO_{Z, z}}\OO_{Y, y}$. 
Additionally, if $X, Y, Z$ are affine, then the fibered product is also affine and $\OO_{X\times_ZY}(X\times_ZY)\simeq \OO_X(X)\otimes_{\OO_Z(Z)}\OO_Y(Y)$.

\begin{remark}
This result is non-trivial and strongly depends on the conditions of schematicity. The fibered product of more general finite spaces (flat restrictions), is not necessarily finite.
\end{remark}

\textit{A priori}, the notions of affine spaces and schematic morphisms seem unrelated, but there is synergy between them. We will refer to \cite{Fernando schemes} when additional results are needed. The most relevant family of affine schematics morphisms are \textit{qc-isomorphisms}.
\begin{definition}
A morphism $f:X\longrightarrow Y$ in $\mathbf{SchFin}$ is said to be a qc-isomorphism (where <<qc>> stands for quasi-coherent) if it is affine and $f_*\OO_X\simeq \OO_Y$.
\end{definition}
Clearly, if $S$ is a scheme any two different finite models of it, $X$ and $Y$, are qc-isomorphic (there is a natural morphism determined by which covering is \textit{thinner} in the sense of the intersections $U^s$, i.e. by the inclusion between the sublattices they generate).
\begin{proposition}\label{prop cohomological characterization qciso}
If $f:X\longrightarrow Y$ is schematic, then it is a qc-isomorphism if and only if it induces an equivalence of categories $(f_*, f^*):\mathbf{Qcoh}(X)\simeq \mathbf{Qcoh}(Y)$ and $\mathrm{R}^if_*\M=0$ for $\M\in\mathbf{Qcoh}(X)$ and $i>0$ (i.e. direct images preserve cohomology of quasi-coherent sheaves).
\begin{proof}
The <<only if>> part is proven in \cite{Fernando schemes}. For the converse: from the equivalence of categories we have $f_*f^*\OO_Y\simeq \OO_Y$, so $f_*\OO_X\simeq \OO_Y$. To prove that $f^{-1}(U_y)$ is affine for every $y\in Y$, we have to check that it is acyclic and that taking global sections induces an equivalence of categories. The first part follows immediately from the condition $\mathrm{R}^if_*\OO_X=0$ for $i>0$. The second part follows from the fact that $U_y$ is affine, that we have $\OO_{Y, y}\simeq \OO_X(f^{-1}(U_y))$ by the previous argument and that the equivalence of the hypothesis restricts to $\mathbf{Qcoh}(f^{-1}(U_y))\simeq \mathbf{Qcoh}(U_y)$ after applying  the extension theorem for quasi-coherent sheaves (Theorem \ref{theorem extension of quasicoherent sheaves}). 
\end{proof}
\end{proposition}
\begin{proposition}[Stein Factorization] \label{stein factorization}
Let $f:X\longrightarrow Y$ be a morphism in $\mathbf{SchFin}$, then $f=\rho\circ f'$ with $\rho:Y':=(Y, f_*\OO_X)\longrightarrow Y$ an affine morphism (the topological identity) and $f'$ such that $f'_*\OO_X=\OO_{Y'}$. Additionally, if $f$ is affine, then $f'$ is a qc-isomorphism.
\end{proposition}
\begin{theorem}
The set of qc-isomorphisms is a multiplicative system of arrows in $\mathbf{SchFin}$, hence there is a localization or homotopy category defined by them, which we denote $\mathbf{SchFin}_{qc}$.
\end{theorem}
 If $\textbf{Schemes}^{qc-qs}$ denotes the category of quasi-compact and quasi-separated schemes, this localization allows us to define the following diagram:
\begin{align*}
\xymatrix{ 
\textbf{Schemes}^{qc-qs}\ar[rr]^{\Psi}\ar[rd]_{Forget} & & \mathbf{SchFin}_{qc}\ar[ld]^{\mathfrak{Spec}}\\
& \textbf{LRS} &
}
\end{align*}
where all the functors involved are {fully faithful}. Cohomology, categories of quasi-coherent sheaves (and other objects we shall not be concerned with here, like derived categories) are stable under qc-isomorphisms, and thus can be studied picking just one representative from each equivalence class. We will show later in this paper that the \'etale fundamental group verifies this property.

The following proposition completes Proposition \ref{stein factorization}.

\begin{proposition}[Relative spectrum]\label{proposition relative spectrum}
Let $X$ be a schematic finite space and $\A$ a quasi-coherent sheaf of algebras. Then $(X, \A)$ is schematic and the natural morphism $(X, \A)\to (X, \OO_X)$ is schematic. This stablishes an equivalence between the category of quasi-coherent algebras on $X$ and the category of affine schematic morphisms to $X$ modulo qc-isomorphisms.
\end{proposition}

In particular, the subcategory of affine schematic finite spaces localized by qc-isomorphisms is \textit{equivalent} to the category of affine schemes. 

\begin{definition}
We say that a schematic finite space $X$ has \textit{open restrictions} if $r_{xy}\colon\OO_{X, x}\to \OO_{X, y}$ is of finite presentation for every $x\leq y$. This condition is equivalent to $r_{xy}^*=\Spec(r_{xy})$ being an open immersion of affine schemes by \cite[IV, 17.9.1]{EGAIV}. We denote by $\mathbf{SchemesFin}$ the subcategory of schematic finite spaces with open restrictions.
\end{definition}
\begin{remark}
Although any finite space $X$ verifying the property of <<having open restrictions>> gives rise to a scheme $\mathfrak{Spec}(X)$, \textit{schematic} finite spaces have a richer structure that forces them to have <<scheme-like>> properties. In this sense, schematic finite spaces are \textit{more} than just ordered descent data or simplicial schemes.
\end{remark}

By construction, $\mathbf{SchemesFin}_{qc}$ is the essential image of $\Psi$ and $\mathfrak{Spec}$ is an inverse of $\Psi$ when restricted to this subcategory, giving an isomorphism of categories. Finally, the essential image of $\mathfrak{Spec}\colon \mathbf{SchFin}_{qc}\to \textbf{LRS}$ is the subcategory of locally ringed spaces that are gluings of affine schemes by flat monomorphisms of schemes, which may be tentatively called <<the category of \textit{protoschemes}>>.
\smallskip

We conclude with an example of a finite schematic space that is not a finite model of a scheme.

\begin{example}\label{example:schematic-not-scheme}
Consider the poset $X=\{x, y, z\}$ such that $z>x$, $z>y$; $\OO_{X, x}=A[x]$, $\OO_{X,y}=A[x]$, $\OO_{X, z}=K(x)$, where $A$ is a domain and $K$ its fraction field. This models the gluing of two affine lines $\mathbb A=\Spec(A[x])$ along the generic point. It is a schematic finite space but not the finite model of any scheme.
\end{example}

\section{Connectedness and Geometric Points}

In this section we start to lay the groundwork for the study of finite \'etale covers on a schematic finite space. Since the topological and algebraic information contained in schematic finite spaces are not as tightly linked as in schemes (or locally ringed spaces), we need to be careful when defining notions of connectedness to be sure that they reflect the true geometry behind the descent data that our space contains.

Additionally and for obvious reasons, we need to examine the notion of <<geometric point>>, defined in a way that is analogue to the notion for schemes, \textit{a priori} different from a topos theoretic approach.

\subsection{Connectedness. Geometric schematic finite spaces.}

The order topology of a schematic finite space is linked to the combinatorics of the idea of <<choosing a covering>> rather than to the algebraic structure. We start with a list of definitions:

\begin{definition}
Let $X$ be a schematic finite space. 
\begin{itemize}

\item  We say that $X$ is \textit{top-connected} (topologically connected) if its underlying poset is connected.

\item  We say that $X$ is \textit{connected} if $\Spec(\OO_X(X))$ is connected.

\item  We say that $X$ is \textit{pw-connected} (pointwise-connected) if $\Spec(\OO_{X, x})$ is connected for every $x\in X$.

\item  We say that $X$ is \textit{well-connected} if it is connected and pointwise-connected.

\end{itemize}
\end{definition}

Recall that the prime spectrum of a non-zero ring is connected if and only if the ring has no non-trivial idempotent elements. If a schematic finite space $X$ is connected, $\mathfrak{Spec}(X)$ is connected in the usual sense, since $X$ and $\mathfrak{Spec}(X)$ have the same global sections. The empty set $\emptyset$ is not considered connected. In particular, a pw-connected space has non-zero stalks.

\begin{proposition}\label{proposition well connected is top + pw}
Let $X$ be a schematic finite space. $X$ is well-connected if and only if it is top-connected and pw-connected.
\begin{proof}
The <<only if>> part is clear noticing that connected implies top-connected. The converse amounts to proving that a connected colimit of non-empty connected topological spaces $\{X_i\}$, is connected. This is well known and follows from the fact that $X$ is endowed with the quotient topology of $\coprod_i X_i\longrightarrow X$, so if $X=C_1\amalg C_2$ for some non-empty closed and open sets $C_1$ and $C_2$, either some $X_i$ is disconnected or the underlying diagram is disconnected, which is a contradiction.
\end{proof}
\end{proposition}

We are mainly interested in well-connected spaces. Next, we see that for any schematic finite space $X$, we can construct a pw-connected space $X'$ and a qc-isomorphism $X'\longrightarrow X$. If we can ensure that $X'$ is top-connected, then $X$ is equivalent to a well connected space. The idea is replacing each point of $X$ by as many points as connected components $\Spec(\OO_{X, x})$ has.
\medskip

In general, let $\pi_0\colon \mathbf{Top}\to \mathbf{Set}$ denote the connected components functor. For any ringed finite topological space $X$ and $x\leq x'\in X$, the restriction $r_{xx'}$ induces a map
\begin{align*}
\pi_0(r_{xx'}^*)\colon\pi_0(\Spec(\OO_{X, x'}))\to\pi_0(\Spec(\OO_{X, x})).
\end{align*}
\begin{definition}
For any such $X$, we define $\mathbf{pw}(X)$ as follows:
\begin{itemize}
\item As a set, $\mathbf{pw}(X)=\coprod_{x\in X}\pi_0(\Spec(\OO_{X, x}))$, which is finite because $X$ is finite.
\item $\mathbf{pw}(X)$ is endowed with the topology defined by the partial order such that, given a pair $y\in\pi_0(\Spec(\OO_{X, x}))$ and $y'\in\pi_0(\Spec(\OO_{X, x'}))$,
\begin{align*}
y\leq y' \Longleftrightarrow x\leq x' \text{ and }\pi_0(r_{xx'}^*)(y')=y.
\end{align*}
\end{itemize}

Since, for every $x\in X$, $\Spec(\OO_{X, x})=\Spec(A_1^x)\amalg...\amalg\Spec(A_{n_x}^x)$ for certain non-zero rings $A_1^x, ..., A_{n_x}^x$ with connected spectrum (note that connected components of locally Noetherian topological spaces are open), we have a bijections $$\phi_x\colon \pi_0(\Spec(\OO_{X, x}))\overset{\sim}{\to} I_x:=\{1, ..., n_x\}.$$
In other words, for every $x\leq x'$ we have maps $I_{x'}\to I_x$. For every $i\in I_{x'}$ with image $j\in I_x$, since the continuous image of a connected set is connected (so it has to be contained in a connected component), the map $r_{xx'}^*\colon\Spec(\OO_{X, x'})\to\Spec(\OO_{X, x'})$ induces a unique morphism of affine schemes $\Spec(A_i^{x'})\to \Spec(A_j^x)$, and thus a unique morphism of rings $A_j^x\to A_i^{x'}$.
\begin{itemize}
\item We endow $\mathbf{pw}(X)$ with the sheaf of rings such that, for every $y\in\pi_0(\Spec(\OO_{X, x}))$ as above,
\begin{align*}
\OO_{\mathbf{pw}(X), y}=A_{\phi_x(y)}^x,
\end{align*}
and, for every $y\leq y'$ with $y'\in\pi_0(\Spec(\OO_{X, x'}))$, 
\begin{align*}
r_{yy'}\colon A_{\phi_x(y)}^x\to A_{\phi_{x'}(y')}^{x'}
\end{align*}
is the unique ring homomorphism constructed in the previous 
discussion.
\item Finally, there is a natural surjective morphism of ringed spaces
\begin{align*}
\pi\colon\mathbf{pw}(X)=\coprod_{x\in X}\pi_0(\Spec(\OO_{X, x}))\to X
\end{align*}
sending every $y\in\pi_0(\Spec(\OO_{X, x}))$ to $x$ and such that $\pi^{\#}_y\colon \OO_{X, x}\to A^x_{\phi_x(y)}$ is the natural projection of the direct product. If $X$ is pw-connected, $\pi$ is clearly the identity. 
\end{itemize}
\end{definition}

This construction is functorial: given a morphism of ringed spaces $f\colon X\to X'$, the ring homomorphisms $f_x^\#\colon \OO_{X', f(x)}\to\OO_{X, x}$ induce natural morphisms
\begin{align*}
\pi_0((f^{\#}_x)^*)\colon\pi_0(\Spec(\OO_{X, x}))\to \pi_0(\Spec(\OO_{X', f(x)}))
\end{align*}
for every $x\in X$. We define a continuous map
\begin{align*}
\mathbf{pw}(f)\colon \mathbf{pw}(X)&\to \mathbf{pw}(X')\\
y&\mapsto \pi_0((f^{\#}_{\pi(y)})^*)(y),
\end{align*}
and endow it with the morphisms of sheaves of rings such that, for every $y\in\mathbf{pw}(X)$,
\begin{align*}
\mathbf{pw}(f)_y^{\#}\colon A^{\pi(y)}_{\phi_{\pi(y)}(y)}\to A^{\pi(\mathbf{pw}(f)(y))}_{\phi_{\pi(\mathbf{pw}(f)(y))}(\pi(\mathbf{pw}(f)(y)))};
\end{align*}
which is unique due to the same argument that made $r_{yy'}$ unique (the continuous image of a connected space is connected). Finally, there is a commutative diagram
\begin{align*}
\xymatrix{ 
\mathbf{pw}(X)\ar[r]^{\mathbf{pw}(f)}\ar[d]_{\pi_X}& \mathbf{pw}(Y)\ar[d]^{\pi_Y}\\
X\ar[r]^f& Y.}
\end{align*}

\begin{remark}\label{remark minimal open pw-conn}
Note that we have an explicit description of the minimal open sets of $\mathbf{pw}(X)$. Namely, for $y\in \mathbf{pw}(X)$ with $\pi(y)=x$, we have the natural morphism $r_x\colon\OO_{X, x}\to \prod_{x'\geq x}\OO_{X, x'}$ inducing $\pi_0(r_x^*)\colon\coprod_{x'\geq x}\pi_0(\Spec(\OO_{X, x'}))\to \pi_0(\Spec(\OO_{X, x}))$, where both spaces are topologized as subspaces of $\mathbf{pw}(X)$ and thus the target is discrete. Now
\begin{align*}
U_y=\pi_0(r_x^*)^{-1}(y).
\end{align*}
In particular, one has that for $y, y'\in\mathbf{pw}(X)$ with $\pi(y)=\pi(y')=x$, $U_y\cap U_{y'}=\emptyset$. In other words, we get a partition \begin{align*}
\pi^{-1}(U_x)=U_{y_1}\amalg...\amalg U_{y_n}.
\end{align*}
This will be key to prove the following result.
\end{remark}
\begin{proposition}\label{proposition pw is schematic}
If $X$ is a schematic finite space, $\mathbf{pw}(X)$ is a schematic finite space and the projection $\pi\colon \mathbf{pw}(X)\to X$ is a schematic qc-isomorphism.
\begin{proof}
$\pi$ is affine because, by Remark \ref{remark minimal open pw-conn}, for every $x$ such that $\pi^{-1}(x)=\{y_1, ..., y_n\}$ we have $\pi^{-1}(U_x)=U_{y_1}\amalg...\amalg U_{y_n}$. Finally, the definition of the structure sheaf of $\mathbf{pw}(X)$ yields $$(\pi_*\OO_{\mathbf{pw}(X)})_x=\OO_{\mathbf{pw}(X)}(U_{y_1}\amalg...\amalg U_{y_n})=\OO_{\mathbf{pw}(X), y_1}\times ...\times \OO_{\mathbf{pw}(X), y_n}\simeq \OO_{X, x}.$$

Assume for now that $\mathbf{pw}(X)$ schematic. If this were the case, to prove that $\pi$ is schematic it would suffice to see that $\mathbb{R}\pi_*\M$ is quasi-coherent for any quasicoherent module $\M$. In other words, that for every $x\leq x'$ in $X$ and every $i$, the natural morphism
\begin{align*}
H^i(\pi^{-1}(U_x), \M)\otimes_{\OO_{X, x}} \OO_{X, x'}\simeq H^i(\pi^{-1}(U_x), \M)
\end{align*}
is an isomorphism. Since all quasi-coherent modules on $\pi^{-1}(U_x)=U_{y_1}\amalg...\amalg U_{y_n}$ are acyclic, the only non-trivial case is $i=0$. We have (since all direct products are finite, hence isomorphic to direct sums of modules!)
\begin{align*}
&\M(U_{y_1}\amalg...\amalg U_{y_n})\otimes_{\OO_{X, x}} \OO_{X, x'}\simeq \bigg(\prod_{j=1}^n\M(U_{y_j})\bigg)\otimes_{\OO_{X, x}} \OO_{X, x'}\simeq \\
&\simeq \prod_{j=1}^n\M(U_{y_j})\otimes_{\OO_{X, x}} \OO_{X, x'}\simeq \prod_{j=1}^n\prod_{y^{'}\in \pi^{-1}(x')\cap U_{y_j}}\M(U_{y^{'}})\simeq \M(\amalg_{y'\in\pi^{-1}(x')}U_{y'});
\end{align*}
where we use that $\M(U_{y_j})\otimes_{\OO_{X, x}}\OO_{X, x'}\simeq \M(U_{y_j})\otimes_{\OO_{\mathbf{pw}(X), y_j}}\big(\prod_{y'\in\pi^{-1}(x')}\OO_{\mathbf{pw}(X), y'}\big)$ and that these products are non-zero only when $y'\geq y_j$.

Finally, we see that $\mathbf{pw}(X)$ is indeed schematic. If $y, y'\in\mathbf{pw}(X)$ verify $\pi(y)=\pi(y')$, then $U_y\cap U_{y'}=\emptyset$ by Remark \ref{remark minimal open pw-conn}. Additionally, $U_{y}\cap U_{y'}=\emptyset$ whenever $U_{\pi(y)}\cap U_{\pi(y')}=\emptyset$, so we only need to check that, for all $y, y'$ with $\pi(y)\neq \pi(y')$, $U_{\pi(y)}\cap U_{\pi(y')}\neq\emptyset$ and any $y''\geq y'$, 
\begin{align*}
H^i(U_y\cap U_{y'}, \OO_{\mathbf{pw}(X)})\otimes_{\OO_{\mathbf{pw}(X), y'}}\OO_{\mathbf{pw}(X), y''}\overset{\sim}{\to} H^i(U_y\cap U_{y''}, \OO_{\mathbf{pw}(X)}).
\end{align*}
To see this, we will compare the standard resolutions of $\OO_{\mathbf{pw}(X)}$ and $\OO_X$ (Equation \ref{equation standard resolution}). Denote all intersections of minimal sets by $U_{a}\cap U_{b}=U_{ab}$ ($a, b$ points), $x=\pi(y)$ and $x'=\pi(y')$. We have $$\pi^{-1}(U_{xx'})=\coprod_{\substack{y_j\in \pi^{-1}(x)\\ y'_k\in \pi^{-1}(x')}}U_{y_jy_k'}.$$ For every $i$, it is straightforward to see that we have the following decomposition compatible with the differentials:
\begin{align*}
(C^i\OO_X)(U_{xx'})=&\prod_{t_i>...>t_0\in U_{xx'}}\OO_{X, t_i}\simeq\prod_{\substack{t_i>...>t_0\in U_{xx'}\\z_i\in\pi^{-1}(t_i)}}\OO_{\mathbf{pw}(X), z_i}=\\
&=\prod_{\substack{y_j\in \pi^{-1}(x)\\ y'_k\in \pi^{-1}(x')}}\prod_{ \substack{ t_i>...>t_0\in U_{xx'}\\z_i\in\pi^{-1}(t_i) \\z_i\in U_{y_jy_k'} }}\OO_{\mathbf{pw}(X), z_i}=\prod_{\substack{y_j\in \pi^{-1}(x)\\ y'_k\in \pi^{-1}(x')}}\prod_{z_i>...>z_0\in U_{y_jy_k'}}\OO_{\mathbf{pw}(X), z_i}=\\&=\prod_{\substack{y_j\in \pi^{-1}(x)\\ y'_k\in \pi^{-1}(x')}}(C^i\OO_{\mathbf{pw}(X)})(U_{y_jy_k'})
\end{align*}
where we have used that, by definition of $\pi\colon\mathbf{pw}(X)\to X$, there is a bijection of sets
\begin{align*}
\{z_i>...>z_0\in U_{y_jy_k'}\}=\{t_i>...>t_0\in U_{xx'}:\pi(z_i)=t_i\text{ and }z_i\in U_{y_jy_k'}\}.
\end{align*}
Indeed, each $z_i>...>z_0$ clearly produces a chain $\pi(z_i)>...>\pi(z_0)$ verifying the conditions. Conversely, given $\pi(z_i)>t_{i-1}>...>t_0$ with $z_i\in U_{y_jy'_k}$, we define $z_j=\pi_0(r_{t_jt_i}^*)(z_i)$ for all $j<i$. These correspondences are mutually inverse.

Now, by hypothesis of $X$ being schematic we have, for all $x''>x'$,
\begin{align*}
 H^i(U_{xx'}, \OO_X)\otimes_{\OO_{X, x'}}\OO_{X, x''}\overset{\sim}{\to}H^i(U_{xx''}, \OO_X)
\end{align*}
and the previous discussion implies that our cohomology modules decompose as
\begin{align*}
& H^i(U_{xx'}, \OO_X)\simeq \prod_{\substack{y_j\in \pi^{-1}(x)\\ y'_k\in \pi^{-1}(x')}}H^i( U_{y_jy_k'}, \OO_{\mathbf{pw}(X)}),\\
 & H^i(U_{xx''}, \OO_X)\simeq\prod_{\substack{y_j\in \pi^{-1}(x)\\ y'_k\in \pi^{-1}(x'')}}H^i( U_{y_jy_k''}, \OO_{\mathbf{pw}(X)});
\end{align*}
and thus, for all $y, y'<y''$ with $\pi(y)=x, \pi(y')=x', \pi(y'')=x''$, isomorphisms
\begin{align*}
H^i( U_{yy'}, \OO_{\mathbf{pw}(X)})\otimes_{\OO_{X, x'}}\OO_{X, x''}\simeq H^i( U_{yy''}, \OO_{\mathbf{pw}(X)}).
\end{align*}
Note that this is exactly the isomorphism we are looking for, because  $\OO_{X, x''}\simeq \prod_{y''\in\pi^{-1}(x')}\OO_{\mathbf{pw}(X), y'}$ and all the non-relevant tensor products are zero as before.
\end{proof}
\end{proposition}
\begin{proposition}
If $f\colon X\to Y$ is a schematic morphism between schematic finite spaces, the induced morphism $\mathbf{pw}(f)\colon\mathbf{pw}(X)\to \mathbf{pw}(Y)$ is schematic.
\begin{proof}
The proof is routinary and similar to the last part of the previous one. If $\pi_X$ and $\pi_Y$ are the corresponding natural projections, we have to see that for all $z<z'\in\mathbf{pw}(X)$, $t<t'\in \mathbf{pw}(Y)$, denoting $g=\mathbf{pw}(f)$, there are isomorphisms
\begin{align*}
&H^i(U_z\cap g^{-1}(U_t), \OO_{\mathbf{pw}(X)})\otimes_{\OO_{\mathbf{pw}(Y), t}}\OO_{\mathbf{pw}(Y), t'}\simeq H^i(U_z\cap g^{-1}(U_{t'}), \OO_{\mathbf{pw}(X)}),\\
&H^i(U_z\cap g^{-1}(U_t), \OO_{\mathbf{pw}(X)})\otimes_{\OO_{\mathbf{pw}(X), z}}\OO_{\mathbf{pw}(X), z'}\simeq H^i(U_{z'}\cap g^{-1}(U_{t}), \OO_{\mathbf{pw}(X)})
\end{align*}
for all $i$. Since $\pi_Y\circ g=f\circ \pi_X$, we have that 
\begin{align*}
\pi_X^{-1}(f^{-1}(U_y))=g^{-1}(\coprod_{t\in\pi_Y^{-1}(U_y)}U_t)=\coprod_{t\in\pi_Y^{-1}(U_y)}g^{-1}(U_t).
\end{align*}
Now, as in Proposition \ref{proposition pw is schematic}, one relates the corresponding standard resolutions using the hypothesis of $f$ being schematic. We leave the details to the reader.
\end{proof}
\end{proposition}

Let $\mathbf{SchFin}^{pw}$ denote the subcategory of pw-connected schematic finite spaces. We have seen that there is a functor
\begin{align*}
\mathbf{pw}\colon \mathbf{SchFin}\to \mathbf{SchFin}^{pw}.
\end{align*}
We compile its properties in the next theorem.

\begin{theorem}\label{theorem pw connectification}
The natural fully faithful inclusion $\mathbf{SchFin}^{pw}\hookrightarrow \mathbf{SchFin}$ has
\begin{align*}
\mathbf{pw}\colon \mathbf{SchFin}\to \mathbf{SchFin}^{pw}
\end{align*}
as a right adjoint. Furthermore, the natural morphism $\mathbf{pw}(X)\to X$ is a qc-isomoprhism for every $X$, and the identity if and only if $X$ is pw-connected.
\begin{proof}
The only remaining part of the proof is the adjunction. Let $Y$ be a pw-connected schematic finite space. Let us see that for every schematic finite space $X$ there is a bijection
\begin{align*}
\Hom_{\mathbf{SchFin}}(Y, X)=\Hom_{\mathbf{SchFin}^{pw}}(Y, \mathbf{pw}(Y)).
\end{align*}
Indeed, given $f\colon Y\to X$ we apply $\mathbf{pw}$ and obtain $Y=\mathbf{pw}(Y)\to X$. Conversely, $g\colon Y\to\mathbf{pw}(X)$ induces $\pi_X\circ g\colon Y\to X$, where $\pi_X\colon \mathbf{pw}(X)\to X$ is the natural projection. 
\end{proof}
\end{theorem}

\subsubsection{Well-connected components}

Let $X$ be a schematic finite space. The number of topological connected components of $\mathbf{pw}(X)$ coincides with the number of connected components of $\Spec(\OO_{X}(X))$ and also with the number of connected components of $\mathfrak{Spec}(X)$. Each one of these connected components is well-connected. Indeed, this is the content of the following Lemma, which follows directly from the definition of $\mathbf{pw}$, Proposition \ref{proposition well connected is top + pw} and the fact that global sections are preserved under qc-isomorphisms.
\begin{lemma}\label{lemma relation connected and geo}
A schematic finite space $X$ is connected if and only if $\mathbf{pw}(X)$ is top-connected (and thus well-connected).  In particular, for any pw-connected schematic finite space, the notions of connectedness, well-connectedness and top-connectedness are equivalent. 
\end{lemma}

\begin{definition}[Well-connected components]\label{definition well connected components} 
The connected components of $\mathbf{pw}(X)$ are called the \textit{well-connected} components of $X$. We denote by $\pi_0^{wc}(X):=\pi_0(\mathbf{pw}(X))$ the set of well-connected components of $X$. 
\end{definition}
\begin{remark}
If $X$ is pw-connected, $\pi_0^{wc}(X)=\pi_0(X)$.
\end{remark}

If $i_k\colon X_k\hookrightarrow \mathbf{pw}(X)\to X$ denotes the natural inclusion of each connected component composed with the projection, there is a decomposition 
\begin{align*}
\OO_X\simeq \prod_{X_k\in\pi_0^{wc}(X)}i_{k*}\OO_{X_k}.
\end{align*}
The sheaf $i_{k*}\OO_{X_k}$ contains all the algebro-geometric information of the component $X_k$; and actually $X_k$ is the topological support of $i_{k*}\OO_{X_k}$ for all $k$. We can generalize this idea to quasi-coherent sheaves of algebras.

Let $X$ be a well-connected schematic finite space and $\A$ a quasi-coherent sheaf of $\OO_X$-algebras. Consider the natural morphism $g\colon \mathbf{pw}(X, \A)\to (X, \A)\to X$ and $X_\A^1\amalg...\amalg X_A^k$ the decomposition of $\mathbf{pw}(X, \A)$ into connected components. Let $g_j\colon X_\A^1\to X$ denote the composition of $g$ with the natural inclusion for $1\leq j\leq k$.  Since $g_j$ is a composition of affine morphisms, it is affine, and thus $\A_j:=g_{j*}\OO_{X^j_\A}$ is a quasi-coherent algebra on $X$. We have obtained:
\begin{definition}[Well-connected components of a sheaf of algebras]\label{definition well connected faithfully flat}
Let $X$ be a well-connected schematic finite space and $\A$ a quasi-coherent sheaf of $\OO_X$-algebras. There is a decomposition $\A\simeq \A_1\times...\times\A_k$ of quasi-coherent $\OO_X$-algebras. The $\A_j$ are called the \textit{well-connected components of $\A$}. By construction $(X, \A_i)$ is connected for every $i$ and thus $\mathbf{pw}(X, \A_i)$ is well-connected. We say that $\A\neq 0$ is connected if for every decomposition $\A\simeq \A_1\times\A_2$, either $\A_1=0$ or $\A_2=0$.
\end{definition}

In conclusion, connectedness is the algebro-geometric property and it is reflected in the topology of pw-connected spaces. The final statement is:
\begin{theorem}\label{theorem connectedness}
A schematic finite space $X$ is well-connected if and only if for every decomposition $X=X_1\amalg X_2$ in $\mathbf{SchFin}$, either $X_1$ or $X_2$ is qc-isomorphic to $\emptyset$.
\begin{proof}
If $X=X_1\amalg X_2$ and $X$ is well-connected, we have $\mathbf{pw}(X)=\mathbf{pw}(X_1)\amalg \mathbf{pw}(X_2)$ with $\mathbf{pw}(X)$ top-connected by Lemma \ref{lemma relation connected and geo}. We conclude that $\mathbf{pw}(X_i)=\emptyset$ for $i=1$ or $i=2$; in other words, that $X_i$ is qc-isomorphic to $\emptyset$.

Conversely, $X$ admits a decomposition $X=(X, \OO_1)\amalg...\amalg(X, \OO_k)$ with $\OO_1, ..., \OO_k$ the well-connected components of $\OO_X$. By the hypothesis, only one of them is non-zero, so we conclude.
\end{proof}
\end{theorem}
\subsection{Geometric and schematic points}

Schematic finite spaces are not \textit{locally} ringed spaces in general (and those are not cases of interest for us). We could say that their points are <<fat>>. 

For the rest of the paper, our spaces will be endowed with a structure morphism $X\longrightarrow (\star, k)$ with $k$ a (base) field and where $\star$ denotes the singleton space. Morphisms will be considered as morphisms over $(\star, k)$. Regardless, many of the ideas that follow generalize to the relative situation.
\begin{definition}
Given a schematic finite space $(X, \OO_X)$, a \textit{geometric point with values in $\Omega$} or an $\Omega$-geometric point is a \textit{schematic} morphism $(\star, \Omega)\longrightarrow (X, \OO_X)$, where $\Omega$ is an algebraically closed field. i.e. a point $\overline{x}\in X^\bullet(\Omega)$ in $\mathbf{SchFin}$.
\end{definition}
\begin{proposition}[Characterization of geometric points]\label{characterization of schematic points}
An arbitrary morphism of ringed spaces $\overline{x}:(\star, \Omega)\longrightarrow X\text{ }\text{ }(\star\mapsto x)$ is a geometric point if and only if the prime ideal $\p:=\ker(\OO_{X, x}\rightarrow \Omega)$ does not <<lift>> to any $x'>x$; i.e. for every $x'>x$ there is no prime $\p'\subseteq \OO_{X, x'}$ such that $r_{xx'}^{-1}(\p')=\p$.
\begin{proof}
It is a simply consequence of the definition of schematic morphism for this particular case. The only non-trivial condition is that $\Omega\otimes_{\OO_{X, x}}\OO_{X, x'}=0$ for every $x<x'$. In other words, the fiber of the morphism of schemes $\Spec(\OO_{X, x'})\longrightarrow\Spec(\OO_{X, x})$ at the point $\Spec(\Omega)\hookrightarrow \Spec(\OO_{X, x})$ is empty. In terms of rings, this is exactly the <<lifting>> condition of the proposition.
\end{proof}
\end{proposition}

We now consider pairs $(x, \p)$, where $x\in X$ and $\p\subseteq \OO_{X, x}$ is a prime ideal; i.e. elements of $\coprod_{x\in X}\Spec(\OO_{X, x})$. If we have two pairs $(x, \p)$, $(x', \p')$, such that $x\leq x'$ and $\p=r^{-1}_{xx'}(\p')$ the natural restriction morphism $\OO_{X, x}\longrightarrow\OO_{X, x'}$ induces a map of residue fields
\begin{align}
(\OO_{X, x})_\p/\p(\OO_{X, x})_\p=:\kappa(x, \p)\longrightarrow\kappa(x', \p'):=(\OO_{X, x'})_{\p'}/\p'(\OO_{X, x'})_{\p'}.
\end{align}
\begin{proposition}\label{proposition residue field well defined}
If $X$ is a schematic finite space, the map $\kappa(x, \p)\longrightarrow\kappa(x', \p')$ induced by any two pairs as described above is an isomorphism.
\begin{proof}
It follows from Propositions \ref{prop restrictions are flat epic} and \ref{proposition flat epic are local iso}.
\end{proof} 
\end{proposition}

We define the following binary relation in the set of these pairs:
\begin{equation}\label{equation equivalence schematic points}
(x, \p)\sim (y, \q)\Longleftrightarrow \exists (z, \mathfrak{r}) \text{ such that }z\geq x, y\text{ and }r_{xz}^{-1}(\mathfrak{r})=\p\text{, }r_{yz}^{-1}(\mathfrak{r})=\q.
\end{equation}

Our next step is proving that this binary relation is in fact an equivalence relation and that it is actually the equivalence relation realizing $
\mathrm{colim}\,\Spec(\OO_{X})$ as a quotient set of $\coprod_{x\in X}\Spec(\OO_{X, x})$. The fact that this description turns out to be so simple strongly relies on $X$ being schematic.
\begin{lemma}\label{lemma binary relation schematic points}
For any $(x, \p)$ and $(y, \q)$ such that $x, y\geq s$ for some $s\in X$ and $r_{sx}^{-1}(\p)=r_{sx}^{-1}(\q)$, we have $(x, \p)\sim (y, \q)$.
\begin{proof}
If there were not any $z\geq x, y$ we would have that $U_x\cap U_y=\emptyset$ and, since $X$ is schematic, $\OO_{X, x}\otimes_{\OO_{X, s}}\OO_{X, y}=0$, which happens if and only if $\Spec(\OO_{X, x})\times_{\Spec(\OO_{X, s})}\Spec(\OO_{X, y})\neq \emptyset$. Recall that, as a set, a fibered product of schemes $S\times_Z T$ is in bijection with the set of quadruplets $(s, t, z, \p)$ with $z$ the image of $s$ and $t$ and $\p\subseteq \kappa(s)\otimes_{\kappa(z)}\kappa(y)$. In our case, this tensor product of residue fields is a field by Proposition \ref{proposition residue field well defined}, thus $(\p, \q, r_{sx}^{-1}(\p), 0)\in \Spec(\OO_{X, x})\times_{\Spec(\OO_{X, s})}\Spec(\OO_{X, y})$ and we reach a contradiction. 

Now, $U_x\cap U_y$ is affine, because it is an open subset of the affine schematic space $U_s$ (and $U_x, U_y$), and it is acyclic, because $H^i(U_x\cap U_y, \OO_X)\simeq H^i(U_x,\OO_X)\otimes_{\OO_{X, z}}\OO_{X, y}\simeq 0$ for $i>0$; so the natural morphism (Proposition \ref{proposition characterization of affines})
\begin{align*}
R^*\colon\coprod_{t\geq x, y}\Spec(\OO_{X, t})\longrightarrow\Spec(\OO_X(U_x\cap U_y))
\end{align*}
is surjective and thus there is a point $z\geq x, y$ and a prime $\mathfrak{r}\in\Spec(\OO_{X, z})$ such that $R^*(\mathfrak{r})=(\p, \q)$. The pair $(z, \mathfrak{r})$ verifies the condition of equation \ref{equation equivalence schematic points} as desired.
\end{proof}
\end{lemma}
\begin{remark} 
The same proof works if we only ask $x, y\in U$ for some affine open $U$. If $X$ itself is affine, this puts tight constraints on its topology.
\end{remark}
\begin{lemma}
The binary relation defined in equation \ref{equation equivalence schematic points} is an equivalence relation.
\begin{proof}
The non-obvious property is transitivity. Assume we have points $(x, \p)\sim (y, \q)\sim (z, \mathfrak{m})$. In particular there exist pairs $(t, \mathfrak{r})$, $(s, \mathfrak{r}')$ such that $t\geq x, y$ and $s\geq y, z$ and verifying the condition on the ideals.

If $t=s$ then it also follows that $\mathfrak{r}=\mathfrak{r}'$, since restriction maps of schematic spaces are flat epimorphisms and in particular induce injections between spectra. If $t\neq s$ it suffices to show that $(t, \mathfrak{r})\sim (s, \mathfrak{r}')$, which holds by Lemma \ref{lemma binary relation schematic points}.
\end{proof}
\end{lemma}

\begin{corollary}
If $(x, \p)\sim (x, \p')$, then $\p=\p'$. In particular, each equivalence class of pairs $(x, \p)$ has a unique maximal representative, in the sense that $x$ is maximal among all the elements of its class. Furthermore, the set of pairs in the equivalence class of $(x, \p)$ defines a closed subset of $X$ (the topological closure of $x$).
\begin{proof}
The first part is proved in part of the proof of the previous Lemma. The last statement is proved by contradiction with Lemma \ref{lemma binary relation schematic points}.
\end{proof}
\end{corollary}

\begin{definition}
A \textit{schematic point} in a schematic finite space $X$ is an \textit{equivalence class} of pairs $(x, \mathfrak{p})$, where $x\in X$ is a topological point and $\mathfrak{p}\subseteq \OO_{X, x}$ is a prime ideal.
\end{definition}
\begin{remark}
To simplify the notation, we will usually denote a schematic point by one of its representative pairs, which we will assume to be the maximal one unless stated otherwise.
\end{remark}

Given a schematic point $((x, \p)$ on $X$ it now makes sense to define its  \textit{residue field} $\kappa(x, \p)$ as
\begin{equation}
\kappa(x, \p):=(\OO_{X, x})_{\p}/\p(\OO_{X, x})_{\p},
\end{equation}
which is well defined modulo isomorphisms  by Proposition \ref{proposition residue field well defined}.

\begin{proposition}\label{puntos geometricos y esquematicos}
Let $X$ be a schematic finite space and let $\Omega$ be an algebraically closed field. There is a correspondence 
\begin{equation*}
\{\text{Morphisms of ringed spaces }(\star, \Omega)\rightarrow X\}\overset{1:1}{\longleftrightarrow} \{\text{Pairs }(x, \p)\text{ and }\kappa(x, \p)\rightarrow\Omega\}
\end{equation*}
that restricts to the schematic category as
\begin{equation*}
X^\bullet(\Omega)\overset{1:1}{\longleftrightarrow} \{\text{Schematic points }(x, \p)\text{ and }\kappa(x, \p)\rightarrow\Omega\}
\end{equation*}
\begin{proof}
For every $\overline{x}:(\star, \Omega)\rightarrow X$ we define $x:=\overline{x}(\star)$, $\p:=\ker(\OO_{X, x}\rightarrow\Omega)$ and the extension is obtained by factorizing the map $\OO_{x, x}\rightarrow\Omega$ through $\kappa(x, \p)$. Conversely, for every schematic point $(x, \p)$ we define $(\star, \kappa((s, \p)))\rightarrow X$ as $\star\mapsto x$ on the level of sets and the natural localization and quotient maps that define $\kappa(x, \p)$ on the level of rings; composing with $(\star, \Omega)\longrightarrow (\star, \kappa(x, \p))$ we obtain the desired morphism.

The second part follows from the fact that every schematic point $(x, \p)$ has a unique maximal representative, so the morphism of ringed spaces constructed using this representative is a geometric point according to Proposition \ref{characterization of schematic points}.
\end{proof}
\end{proposition}

\begin{remark}
This proposition allows us to use the notation $\kappa(\overline{x})$ for the residue field of a geometric point $\overline{x}\in X^\bullet(\Omega)$. It also allows us to give notions of \textit{closed} and \textit{rational} schematic points, analogue to those of schemes.
\end{remark}

Finally, we will see that schematic points are indeed the analogous notion to points in the sense of schemes. Let $S$ be a quasi-compact and quasi-separated scheme, $\mathcal{U}$ a finite and locally affine finite covering of $S$ and $\pi:S\longrightarrow X$ the quotient map to the corresponding finite model.

\begin{proposition}\label{proposition points scheme are schematic points of the model}
There is a 1:1 correspondence between topological points of $S$ and schematic points of its finite model $X$. 
\begin{align*}
|S|\overset{1:1}{\longleftrightarrow}\{\text{Schematic points }(x, \p)\}
\end{align*}

This correspondence extends to geometric points, i.e. for any algebraically closed field $\Omega$, the quotient map induces a bijection
\begin{align*}
S^\bullet(\Omega)\overset{\sim}{\longrightarrow}X^\bullet(\Omega)
\end{align*}
such that the corresponding residue fields of points on both sides coincide.
\begin{proof}
Let $s\in S$ be a point and $\p_s\in \Spec(\OO_S(U^s))$ the ideal it defines as an element of the affine open $U^s$. Then $(\pi(s), \p_s)$ is a schematic point of $X$.

Conversely, given a schematic point $(x, \p_x)$ of $X$ we have  that $\p_x\subseteq \OO_X(U_x)=\OO_S(U^s)$ for any $s\in \pi^{-1}(x)$. This ideal clearly determines a point of $S$, which is independent of the chosen representative of $(x, \p_x)$. These correspondences are mutually inverse.

From the definitions it also follows that $\kappa(s)=\kappa(\pi(s), \p_s)$.

The second part of the statement follows immediately from the first part and the correspondence between geometric and schematic points with a field extension. The map is given explicitly by simply taking finite models (since $\Omega$ is a field, we can only take the total covering). The inverse map is given by the $\mathfrak{Spec}$ functor.
\end{proof}
\end{proposition}
Although it might not be completely clear what the <<geometric points>> of the ringed space $\mathfrak{Spec}(X)=
\mathrm{colim}\,\Spec(\OO_{X})$ should be for a general schematic finite space $X$ (and to treat it properly, we should talk about sites and topoi); it is clear that as a set, they consist exactly of the collection of schematic points of $X$:
\begin{proposition}\label{proposition bijection schematic points}
If $X$ is an arbitrary schematic finite space, there is a bijection:
\begin{align*}
|\mathfrak{Spec}(X)|\overset{1:1}{\longleftrightarrow}\{\text{Schematic points }(x, \p)\}.
\end{align*}
\end{proposition}
An easy corollary is that different finite models of the same scheme have the same geometric points but, actually, the following general fact holds:
\begin{theorem}\label{theorem functorial qc-isomorphisms}
Let $X, Y$ be two schematic finite spaces over $k$ and let $f:X\longrightarrow Y$ be a schematic morphism. If $f$ is a qc-isomorphism then, for every algebraically closed field extension $k\hookrightarrow \Omega$,  the natural morphism $X^\bullet(\Omega)\longrightarrow Y^\bullet(\Omega)$ is an isomorphism.
\begin{proof}
It suffices to prove that there is a bijection:
\begin{align*}
\Psi:\{\text{Schematic points of }X\}&\longrightarrow\{\text{Schematic points of }Y\}\\
(x, \p)&\longrightarrow (f(x), (f_x^\#)^{-1}(\p))
\end{align*}
By Proposition \ref{proposition bijection schematic points}, this is just the map $|\mathfrak{Spec}(f)|\colon |\mathfrak{Spec}(X)|\to|\mathfrak{Spec}(Y)|$, which is known to be a bijection when $f$ is a qc-isomorphism (see Proposition 6.6 of \cite{Fernando schemes}, where qc-isomorphisms are called \textit{weak equivalences}).
\end{proof}
\end{theorem}
\begin{remark}\label{remark flat immersions}
For any map of ringed finite spaces $f\colon X\to Y$ we define its \textit{cylinder}, as a set, as $\mathrm{Cyl}(f):=X\amalg Y$. Each subspace $X$ and $Y$ inherits the order of the original spaces and we say that $y\leq x$ (with $x\in X$, $y\in Y$) if $f(x)=y$; which turns it into a finite poset.  The ringed structure is inherited from $X$ and $Y$, defining $r_{yx}:=f^\#_x$ for $y\leq x$ with $x\in X$, $y\in Y$. See \cite[Section 4.2]{fernandoypedro}.

Let us define a \textit{flat immersion} of schematic finite spaces $f\colon X\to Y$ as a flat schematic morphism such that its relative diagonal is a qc-isomorphism. This is the analogue of a flat monomorphism of schemes in the schematic context. It is possible to prove that $f$ is a flat immersion if and only if $\text{Cyl}(f)$ is schematic. This fact along with Lemma \ref{lemma binary relation schematic points} can be used to give an elementary proof of the previous theorem, and also to prove that if $f$ is a flat immersion, then $X^\bullet(\Omega)\hookrightarrow Y^\bullet(\Omega)$ is injective for every algebraically closed field $\Omega$, as one would expect from <<monomorphisms>>.
\end{remark}

\begin{remark}
The definition of schematic points, and by extension $\mathfrak{Spec}(X)$, bears some resemblance with the <<$\Spec$>> functor from \textit{primed} ringed spaces to locally ringed spaces constructed in \cite{Gillam} (aiming to compare fibered products of schemes and locally ringed spaces with fibered products of general ringed spaces) and, for ringed topoi and the \textit{terminal prime system}, in \cite[IV.1]{Hakim}. These authors topologize $\coprod_{x\in X}\Spec(\OO_{X, x})$ in a different manner and endow it with a sheaf of rings whose sections are essentially localizations of the stalks $\OO_{X, x}$ satisfying certain compatibility condition. Their construction clearly coincides with ours when $X=\star$ (in that case, it is just the ordinary $\Spec$ functor). Comparing their functor with $\mathfrak{Spec}$ in the finite case might be insightful, specially because we know that for $X_1, X_2\to Y$ schematic,
\begin{align*}
\mathfrak{Spec}(X_1)\times_{\mathfrak{Spec}(Y)}\mathfrak{Spec}(X_2)\overset{\sim}{\to}\mathfrak{Spec}(X_1\times_Y X_2)
\end{align*}
where the first fibered product coincides with the product of locally ringed spaces (or schemes) and the second one is the fibered product ringed spaces.
\end{remark}

\subsubsection{Fibers of a morphism.}
Consider  $f:X\longrightarrow Y$, a schematic morphism between schematic finite spaces. Given a schematic point $(y, \p)$ on $Y$ we can define the \textit{schematic fiber} over it, as a set, as
\begin{equation}
f^{-1}((y, \p))=\{(x, \mathfrak{q}):(f(x), (f_x^\#)^{-1}(\q))\sim (y, \p)\}\subseteq |\mathfrak{Spec}(X)|,
\end{equation}
where $f_x^\#:\OO_{Y, f(x)}\longrightarrow\OO_{X, x}$. By transitivity of the equivalence relation it is easy to see that this does not depend on the representative $(y, \p)$ that we choose. In other words:
\begin{align*}
f^{-1}((y, \p))=|\mathfrak{Spec}(f)|^{-1}((y, \p)).
\end{align*}

\begin{remark}Since each schematic point has a unique maximal representative, the fiber $f^{-1}((y, \p))$ is determined by a union of closures $C_{x_1}\cup...\cup C_{x_n}$ and a set of prime ideals on each $\OO_{X, x_i}$: the fibers of $\p$ via each $\Spec(\OO_{X, x_i})\rightarrow\Spec(\OO_{S, s})$. 
\end{remark}

This definition is compatible with the obvious one for geometric points:
\begin{definition}\label{definition geometric fiber}
Let $f\colon X\to Y$ be a schematic morphism and $(\star, \Omega)\overset{\overline{y}}{\longrightarrow}Y$ a geometric point. \textit{The geometric fiber of $f$ at $\overline{y}$} is $f^{-1}(\overline{y}):=\textbf{pw}((\star, \Omega)\times_YX)$.
\end{definition}

\begin{remark}
If $X$ and $Y$ are pw-connected, the appearance of the functor $\mathbf{pw}$ in the definition simply means that we are considering fibered products in $\mathbf{SchFin}^{pw}$ rather than in $\mathbf{SchFin}$. The technical relevance of this will be apparent in Proposition \ref{proposition fibers of etale}.
\end{remark}

\begin{lemma}
Let $f\colon X\to Y$ be a schematic morphism, $\overline{y}\in Y^\bullet(\Omega)$ a geometric point and $(y, \p)$ its corresponding schematic point. There is a 1:1 correspondence
\begin{align*}
\{\text{Pairs }(x, \q):f(x)=y\text{ and }(f^\#_x)^{-1}(\q)=\p\}\overset{1:1}{\longleftrightarrow}\{\text{Pairs }(x, \q)\text{ of }f^{-1}(\overline{y})\}
\end{align*}
that descends to the quotient as
$$f^{-1}((y, \p))\overset{1:1}{\longleftrightarrow}|\mathfrak{Spec}(f^{-1}(\overline{y}))|$$
\begin{proof}
First, notice that given $x\in X$ with $f(x)=y$, $\OO_{X, x}\otimes_{\OO_{Y, y}}\kappa(y, \p)\neq 0$ if and only if there exists an ideal $\q\in\Spec(\OO_{X, x})$ such that $(f_x^\#)^{-1}(\q)=\p$.

The local ring of $(\star, \kappa(y, \p))\times_YX$ at a point $x$ is $(\OO_{X, x})_\p/\p(\OO_{X, x})_\p$, whose spectrum is the set of primes $\q\subseteq \OO_{X, x}$ lying over $\p$. Since the functor $\textbf{pw}$ removes all points with whose stalk is zero and $f^{-1}(\overline{y})\to (\star, \kappa(y, \p))\times_YX$ is qc-isomorphism, we conclude by Theorem \ref{theorem functorial qc-isomorphisms}.
\end{proof}
\end{lemma}

\begin{corollary}
Given two pairs $(y, \p)$, $(y', \p')$ belonging to the same equivalence class (and assuming $y\leq y'$), there is a bijection
$$f^{-1}((y, \p))\longrightarrow f^{-1}((y', \p')).$$
\end{corollary}

\subsubsection{Categorical behavior:}\label{subsection fiber functors}
Consider a morphism $f:X\longrightarrow Y$ of schematic finite spaces and an algebraically closed field $\Omega$. We have a well defined <<geometric fiber functor>> (naively understanding $Y^\bullet(\Omega)$ as the category whose objects are the set of geometric points and has no non-trivial arrows):
\begin{align}\label{definition geometric fiber functor}
\mathrm{Fib}^f:Y^\bullet(\Omega) &\longrightarrow \mathbf{SchFin}^{pw}\\
[\overline{y}:(\star, \Omega)\rightarrow Y]&\longmapsto f^{-1}(\overline{y})\nonumber
\end{align}

Additionally, if $X\longrightarrow Y$ is the finite model of a morphism of schemes $T\longrightarrow S$ for some suitable coverings, the previous discussion about geometric points identifies $Y^\bullet(\Omega)\simeq S^\bullet(\Omega)$ and we have a commutative diagram
\begin{align*}
\xymatrixcolsep{1in}\xymatrix{ 
S^\bullet(\Omega)\ar[r]^{\mathbf{Fib}^f}\ar@{=}[d]_{\mathfrak{Spec}}& \textbf{Schemes}\\
Y^\bullet(\Omega)\ar[r]^{\mathrm{Fib}^f} & \mathbf{SchFin}^{pw}\ar[u]_{\mathfrak{Spec}}
}
\end{align*}
where ${\mathbf{Fib}}^f:S^\bullet(\Omega)\longrightarrow \textbf{Schemes}$ ($(\overline{s}:\Spec(\Omega)\rightarrow Y)\mapsto \Spec(\Omega)\times_ST$). 

Note that this functor is the restriction of a bifunctor
\begin{align}
\mathrm{Fib}:\mathbf{SchFin}_{/Y}\times Y^\bullet(\Omega)&\longrightarrow \mathbf{SchFin}^{pw}\\
\nonumber \big((f:X\rightarrow Y), \overline{y}\big)&\longmapsto f^{-1}(\overline{y})
\end{align}
where $\mathbf{SchFin}_{/Y}$ denotes the category of schematic finite spaces over $Y$. Now, let us denote by $\text{F}_{\overline{y}}\colon \mathbf{SchFin}_{/Y}\to \mathbf{SchFin}^{pw}$ the restriction of $\mathrm{Fib}$ to a fixed $\overline{y}\in Y^\bullet(\Omega)$. 

If for $f:X\longrightarrow Y$, $f':X'\longrightarrow Y$ there is a qc-isomorphism $\phi:X\longrightarrow X'$ such that $f'\circ \phi=f$, we obtain a qc-isomorphism $f^{-1}(\overline{y})\to f'^{-1}(\overline{y})$ (by stability under base change). In particular this functor factors through localization by qc-isomorphisms as follows (abusing notation):
\begin{align*}
\mathrm{Fib}:(\mathbf{SchFin}_{/Y})_{qc}\times Y^\bullet(\Omega)&\longrightarrow \mathbf{SchFin}_{qc}^{pw}\hookrightarrow \textbf{Schemes}^{qc-qs}\\
\nonumber \big((f:X\rightarrow Y), \overline{y}\big)&\longmapsto f^{-1}(\overline{y}).
\end{align*}

Fixing a geometric point $\overline{y}\in Y^\bullet(\Omega)\simeq S^\bullet(\Omega)$, we obtain a commutative diagram:
\begin{align}\label{diagrama functores fibra}
\xymatrixcolsep{1in}\xymatrix{
\textbf{Schemes}_{/S} \ar[r]^{\textbf{\text{F}}_{\overline{s}}}\ar[d]& \textbf{Schemes}\\
(\mathbf{SchFin}_{/Y})_{qc}\ar[ru]_{\mathfrak{Spec}\text{ }\circ\text{ }\text{F}_{\overline{y}}} &
},
\end{align}
where $\textbf{F}_{\overline{s}}$ is the functor such that $(T\to S)\mapsto \Spec(\Omega)\times_ST$. This construction is also functorial on $S$ due to the standard properties of the fiber product. 

\begin{remark}
In future work, we will consider the fundamental groupoid of $X$ as a category whose objects are the geometric points of $X$. This section will have an analogue in that situation.
\end{remark}

\section{The Category of Finite \'Etale Covers}

There are several equivalent definitions of \'etale morphisms of rings. 

\begin{definition}
A morphism of rings $A\to B$ is said to be \textit{weakly \'etale} (or \textit{absolutely flat}) if it is flat and its multiplication map $B\otimes_AB\to B$ is flat. It is said to be \textit{\'etale} (or that $B$ is an \'etale $A$-algebra) if, furthermore, it is of finite presentation. 
\end{definition}

This definition is equivalent to more classical ones by \cite[Proposition 2.3.3]{scholze}, namely:
\begin{proposition}\label{proposition differential characterization etale}
A morphism of rings $A\to B$ is \'etale if and only if it is of finite presentation and its naive cotangent complex is quasi-isomorphic to zero (equivalently, it is smooth and $\Omega_{A|B}=0$).
\end{proposition}

\begin{proposition}\cite[17.6.2]{EGAIV}
\label{proposition etale algebras over fields}
Let $k$ be a field. A $k$-algebra $A$ is \'etale (or separable) if and only if $A\simeq \prod_{i\in I}K_i$ for some finite set $I$ and each $k\hookrightarrow K_i$ is a finite separable field extension. Furthermore, $A$ is \'etale if and only if $A\otimes_k\bar k\simeq \prod_I\bar k$ for a finite set $I$, if and only if $A\otimes_k\bar k$ is reduced.
\end{proposition}

We may also define:
\begin{definition}
We say that a ring homomorphism $A\to B$ is \textit{pointwise-\'etale} if, for every prime $\p\subset A$, $B\otimes_A\kappa(\p)$ is an \'etale $\kappa(\p)$-algebra.
\end{definition}
\begin{proposition}\label{proposition punctually etale}\cite[I, 5.9]{SGAI}
A ring homomorphism $A\to B$ is \'etale if and only if it is flat, of finite presentation and pointwise-\'etale.
\end{proposition}

\begin{remark}\label{remark local freeness in the notherian case}
If $A$ is Noetherian, an $A$-algebra is flat and of finite presentation if and only if it is finite and locally free. This will be assumed to be the case throughout this paper.
\end{remark}
Now we recall the definition of \'etale morphism in terms of \'etale ring homomorphisms:
\begin{definition}
A morphism of schemes  $f\colon T\to S$ is said to be \'etale if for all affine open subsets $U\subseteq S$ and $V\subseteq T$ such that $U\subseteq f^{-1}(V)$, the natural ring homomorphism $\OO_S(U)\to \OO_T(V)$ is \'etale.
\end{definition}

It is well known that the fibers of an \'etale morphism $f$ are disjoint unions of spectra of finite separable field extension of the residue field (which actually characterizes \'etale morphisms among all flat morphisms that are of finite presentation). Clearly, if $f$ has compact fibers, these disjoint unions are finite. In particular, finite morphisms of schemes have compact fibers. 
\medskip

Let $S$ be a scheme and denote by $\textbf{Fet}_S$ the category of finite and \'etale morphisms to $S$, also known as the category of \textit{finite \'etale covers} of $S$, which is a full and faithful subcategory of $\textbf{Schemes}_{/S}$; and let $\mathbf{Qcoh}^{alg}(S)$ denote the category of quasi-coherent $\OO_S$-algebras (we shall use this notation for \textit{any} ringed space $S$ with its natural topology).  The classical relative spectrum functor $$\underline{\Spec}_S^{\mathrm{op}}\colon\mathbf{Qcoh}^{alg}(S)\to \textbf{Schemes}_{/S}^{\text{Affine}}$$ induces an isomorphism of categories between $\mathbf{Qcoh}^{alg}(S)$ and the category of affine morphisms to $S$, whose inverse is given by $(f\colon T\to S)\mapsto f_*\OO_T$. Since finite morphisms are affine, this allows us to consider a fully faithful functor
\begin{align}
\textbf{Fet}_S^{\mathrm{op}}\to \mathbf{Qcoh}^{alg}(S).
\end{align}
Let $\mathbf{Qcoh}^{fet}(S)$ denote the essential image of this functor. From the definition of \'etale morphisms of schemes and the local characterization of \'etale algebras we obtain: 
\begin{lemma}\label{lemma characterization etale algebras}
Let $S$ be a scheme, a quasi-coherent algebra $\A$ lies in $\mathbf{Qcoh}^{fet}(S)$ if and only if it is finite, flat and for every $s\in S$, $\A_s\otimes_{\OO_{S, s}}\kappa(s)$ is a finite \'etale $\kappa(s)$-algebra. If this condition holds, then for every affine open $U\subseteq S$, $\OO_S(U)\to\A(U)$ is \'etale.
\end{lemma}

\begin{remark}
If $S$ is a scheme, it is well known that $\mathbf{Qcoh}^{fet}(S)\simeq \textbf{Fet}_S\simeq \textbf{Loc}(S_{et})$, where the latter is the category of finite, locally constant sheaves of sets on the small \'etale site of $S$ (which may also be considered as sheaves of $\Z$-modules). The second equivalence is given by the functor of points (on $S_{et}$) and descent. In this paper, we will not be concerned with the study of Grothendieck topologies \textit{directly} on schematic finite spaces, but we will partially bring the last point of view back in Subsection \ref{subsection locally trivial}.
\end{remark}

\begin{definition}
Let $X$ be a schematic finite space. A quasi-coherent $\OO_X$-algebra $\A\in\mathbf{Qcoh}^{alg}(X)$ is an \textit{\'etale cover sheaf} (or simply an \textit{\'etale cover}) if it is finite, flat, and for each schematic point $(x, \p)$, $\A_x\otimes_{\OO_{X, x}}\kappa(x, \p)$ is a finite \'etale $\kappa(x, \p)$-algebra. Equivalently (by Proposition \ref{proposition punctually etale}), if $\OO_{X, x}\to\A_x$ is \'etale for every $x\in X$.
\end{definition}

Propositions \ref{prop restrictions are flat epic} and \ref{proposition flat epic are local iso} guarantee that this notion is well defined. Note that, since we are assuming quasi-coherence, \ref{proposition: characterization qcoh, ftype, coh} says that finiteness is a condition at stalks.

\smallskip

Given a schematic finite space $X$, let $\mathbf{Qcoh}^{fet}(X)$ denote the subcategory of $\mathbf{Qcoh}^{alg}(X)$ whose objects are \'etale cover sheaves. 

\begin{remark}
Perhaps, one might want to define a schematic morphism $f\colon X\to Y$ to be \'etale if $f_x^\#\colon \mathcal O_{Y,f(x)}\to\mathcal O_{X,x}$ is an \'etale ring map for all $x$. If $f$ is affine (automatically finite), $f_*\OO_X$ should be an \'etale cover sheaf. The category of (finite) affine \'etale schematic morphisms over $X$ localized by qc-isomorphisms would then be equivalent to $\mathbf{Qcoh}^{fet}(X)$.

However, there is an issue with this potential definition---that will not affect our discussion here, but is one of the reasons why we prefer to simply state everything in sheaf-theoretic terms---: since the restriction maps of schematic finite spaces are not of finite presentation (in particular, they are not \'etale), the adjoint $f_{\#}\colon \OO_Y\to f_*\OO_X$ may fail to preserve \'etaleness at stalks. On the other hand, the fact that said restriction morphisms are \textit{weakly \'etale} suggests that <<weakly \'etale schematic morphisms>> will be the natural class of maps one would like to work with. In particular, it should be possible to generalize the pro-\'etale fundamental group of schemes (\cite{scholze}) to schematic finite spaces. Moreover, this will be a more natural and general version: the hypothetic (small) <<schematic pro-\'etale site>> of a finite model of schemes would potentially be larger than the pro-\'etale site of the scheme itself. In order to study this site and topos, one needs more sophisticated technology that is beyond our aim in this paper.
\end{remark}

\begin{theorem}
Let $S$ be a quasi-compact and quasi-separated scheme, $\mathcal{U}$ a locally affine finite covering and let $\pi:S\longrightarrow X$ be the projection to the corresponding finite model $X$. The equivalence $(\pi^*, \pi_*)\colon\mathbf{Qcoh}(S)\simeq \mathbf{Qcoh}(X)$ restricts to an equivalence
\begin{align*}
\mathbf{Qcoh}^{fet}(S)\simeq \mathbf{Qcoh}^{fet}(X).
\end{align*}
\begin{proof}
Consider $\B\in\mathbf{Qcoh}^{fet}(S)$. Since $\OO_{X, x}\to (\pi_*\B)_x$ is just $\OO_S(U^s)\to \B(U^s)$ for some $s$ with $\pi(s)=x$, with $U^s$ affine; $\pi_*\B$ is an \'etale cover sheaf by Lemma \ref{lemma characterization etale algebras}.

Conversely, take $\A\in \mathbf{Qcoh}^{fet}(X)$. Then $\pi^*\A$ is flat because flatness is a local condition and for every $s\in S$,  $(\pi^*\mathcal{A})_y=\mathcal{A}_{\pi(s)}\otimes_{\OO_{S}(U^s)}\OO_{S, s}$, so (since $\OO_S(U^s)\rightarrow \OO_{S, s}$ is a localization and thus flat) it is sufficient for $\A_{\pi(s)}$ to be a flat $\OO_{X, \pi(s)}$-algebra, which is true by hypothesis. Finiteness follows from the fact that it can be checked for the open cover $\{U^s\}_{s\in S}$, and there, since $\pi(U^s)=U_{\pi(s)}$ and $U^s$ is an affine open, we have that $(\pi^{*}\mathcal{A})(U^s)=\mathcal{A}_{\pi(s)}\otimes_{\OO_{X, \pi(s)}}\OO_{S}(U^s)\simeq \mathcal{A}_{\pi(s)}$, which is a finite $\OO_{X, \pi(s)}$-algebra by hypothesis.

Finally, let $s\in S$ be any point, $\p_s$ be the prime ideal it defines on $U^s$ and $x=\pi(s)$.
\begin{align*}
(\mathcal{\pi^*\mathcal{A}})_s\otimes_{\OO_{S, s}}\kappa(s)=\mathcal{A}_x\otimes_{\OO_S(U^s)}\OO_{S, s}\otimes_{\OO_{S, s}}\kappa(x, \p_s)=\mathcal{A}_x\otimes_{\OO_{X, x}}\kappa(x, \p_s),
\end{align*}
concludes the proof in virtue of Lemma \ref{lemma characterization etale algebras}.
\end{proof}
\end{theorem}

\subsection{Finite locally free sheaves on schematic finite spaces}

Let $A\longrightarrow B$ be a morphism of rings such that $B$ is finitely generated and locally free over $A$. Once again, recall that if $A$ is Noetherian, this is equivalent to $B$ being finite and flat. The following definition is classical.

\begin{definition}
The degree (or rank) of $A\longrightarrow B$ (or simply $B$ if the $A$-algebra structure is understood) is the map defined as
\begin{align*}
\mathrm{deg}(B):\Spec(A)&\longrightarrow\Z^+\\
\p&\longmapsto \text{rank}_{A_{\p}}B_{\p},
\end{align*}
which is continuous and locally constant, hence constant if $\Spec(A)$ is connected. If $\Spec(A)$ is connected and non-empty, we identify $\mathrm{deg}(B)\in\Z^+$.
\end{definition}
\begin{remark}
We always assume that the base ring $A$ is non-zero (otherwise $B=0$ as well). In this case it is easy to see that $\mathrm{deg}(B)=0$ if and only if $B=0$, $\mathrm{deg}(B)\geq 1$ if and only if $A\longrightarrow B$ is injective, and $\mathrm{deg}(B)=1$ if and only if $A\simeq B$.
\end{remark}

Now, let $X$ be a schematic finite space. If $\A\in\mathbf{Qcoh}^{alg}(X)$, we know that $(X, \A)$ is schematic (by Proposition \ref{proposition relative spectrum}), hence the restriction maps $\A_x\to \A_y$ ($x\leq y$) are local isomorphisms of rings. It turns out that they are also local isomorphisms of modules with respect to localizations at ideals of the base rings:
\begin{lemma}\label{lemma restrictions between quasi-coherent are local iso}
Let $X$ be a schematic finite space and $\A$ a quasi-coherent algebra on $X$. For every $x\leq y$ and each prime ideal $\p\in\Spec(\OO_y)$,  one has
\begin{align*}
(\A_x)_{r_{xy}^{-1}(\q)}\simeq (\A_y)_\p,
\end{align*}
where $r_{xy}\colon \OO_{X, x}\to\OO_{X, y}$ is the restriction map.
\begin{proof}
Consider $i:=r_{xy}^*:\Spec(\OO_{X, y})\hookrightarrow\Spec(\OO_{X,x})$. Since $X$ is schematic, we have that $i^{-1}\widetilde{\OO_{X, x}}\simeq \widetilde{\OO_{X, y}}$ (Proposition \ref{proposition flat epic are local iso}). In particular, considering the quasi-coherent modules associated to $\A_x$ and $\A_y$, we obtain that 
\begin{align*}
\widetilde{\A_y}=\widetilde{\A_x\otimes_{\OO_x}\OO_y}=i^*\widetilde{\A_x}=i^{-1}\widetilde{\A_x},
\end{align*}
which proves the claim. 
\end{proof}
\end{lemma}

\begin{remark}If we define the localization of $\A$ at a schematic point $(x, \p)$ to be $(\A_x)_\p$, this lemma guarantees that the notion is well defined within the equivalence class.
\end{remark}
\medskip

Now let $\mathcal{A}$ be a finite and locally free (quasi-coherent) sheaf of $\OO_X$-algebras.

\begin{proposition}
If $X$ is pw-connected, there is a well defined notion of degree of the finite and locally free sheaf $\A$ as a locally constant function.
\begin{proof}
We want to see that there is a well defined map
\begin{align*}
\mathrm{deg}(\A):X&\longrightarrow\Z^+\\
x&\longmapsto \mathrm{deg}(\A_x).
\end{align*}
Indeed, we have to prove $\mathrm{deg}(\A_x)=\mathrm{deg}(\A_y)$ for any two points $x, y$ in the same connected component. By transitivity it is enough to see it for $x\leq y$, and in this case the result follows from Lemma \ref{lemma restrictions between quasi-coherent are local iso} and the hypothesis on $X$.
\end{proof}
\end{proposition}
In particular, if $X$ is connected, the degree of $\A$ is constant and identified with a non-negative integer $\mathrm{deg}(\A)\in\Z^+$. 

\subsection{\'Etale cover sheaves are locally trivial}\label{subsection locally trivial}

Our \'etale sheaves do verify the property of local triviality equivalent to that of \'etale covers of schemes. In order to prove it we use as a guide the ideas of Lenstra,  see \cite[
Proposition 5.2.9, p. 159]{Szamuely}. First, let us introduce the necessary terminology to properly describe this phenomena. 

\begin{definition}
Let $X$ be a schematic finite space and let $\mathbf{Qcoh}^{alg}(X)$ be its category of quasi-coherent algebras. We say that a quasi-coherent algebra $\B$ is a <<\textit{covering}>> if it is finite and faithfully flat ($\OO_{X, x}\to\B_x$ faithfully flat for all $x\in X)$. Let $\mathbf{Cov}(\mathbf{Qcoh}^{alg}(X))$ the family of all such coverings.  Equipped with this family of coverings, $\mathbf{Qcoh}^{alg}(X)$ defines a \textit{site}, which we denote by $X_{\mathbf{Qcoh}}^{fppf}$.
\end{definition}
\begin{remark}
The reader might have noticed that the notion of site we are tacitly using is a <<contravariant>> version of the one that appears in most standard literature on the subject, sometimes called \textit{cosite} we are defining the coverings of $\A$ to be arrows \textit{from} $\A$ instead of \textit{to} $\A$; but have chosen to avoid that less standard term, because if one wants to stick with the ordinary conventions it is sufficient to consider the opposite category of sheaves. Of course, fibered products in this opposite category are tensor products of algebras, coproducts are products of algebras and so on. Regardless of these specific considerations, it is clear that $X_{\mathbf{Qcoh}}^{fppf}$ is the analogue of the fppf site of (Noetherian) schemes. Replacing <<finite and faithfully flat>> by <<\'etale>> would also work for our discussion.  
\end{remark}
Given a schematic morphism $f\colon X\to Y$, we have adjoint functors $$(f_*, f^*)\colon \mathbf{Qcoh}^{alg}(X)\overset{\sim}{\to}\mathbf{Qcoh}^{alg}(Y).$$ In particular, given a quasi-coherent $\OO_X$-algebra $\A$, the natural schematic morphism $f\colon (X, \A)\to X$ induces a functor:
\begin{align*}
f^*\colon X_{\mathbf{Qcoh}}^{fppf}&\to \A_{\mathbf{Qcoh}}^{fppf}:=\mathbf{pw}(X, \A)_{\mathbf{Qcoh}}^{fppf}\\
\B&\mapsto \B\otimes_{\OO_{X}}\A.
\end{align*}
\begin{remark}For simplicity we omit the pullback of $\B$ by $\mathbf{pw}(X, \A)\to (X, \A)$ from the notation, since it induces an equivalence of categories of quasi-coherent algebras and, as we will see in the next section, of sites.
\end{remark}

Let us elaborate on how the usual notions of constant and locally constant objects on a site translate to our case.

Let $X$ be a schematic finite space and let $\OO_1, ..., \OO_k$ be the well-connected components of $\OO_X$ (see Definitions \ref{definition well connected components}, \ref{definition well connected faithfully flat}). Given a finite set $F$ with cardinality $n$, consider the constant functor:
\begin{align*}
\underline{F}\colon X_{\textbf{Qcoh}}&\to \mathbf{Set}\\
\A&\mapsto F
\end{align*}
and its sheafification, denoted $\underline{F}^\#$, which sends $\A$ to $\coprod_{\pi_0^{wc}(X, \A)}F$, where $\pi_0^{wc}(X, \A)$ is the set of well-connected components (Definition \ref{definition well connected components}) of $(X, \A)$. Just as in the case of constant covers in ordinary theories, it follows that $\underline{F}^\#$ is representable (in the sense of the covariant functor of points) by $\prod_{1\leq j\leq k}\prod_F\OO_j=\OO_1^{\times n}\times...\times\OO_k^{\times n}\simeq \OO_X^{\times n}$. This representative is called \textit{constant object} of cardinality $F$ (or degree $n=\#F$).

\begin{remark}\label{remark remove localization}
If $X$ is pw-connected, $\mathbf{pw}(X, \OO_X^{\times n})\to X$ is just $X\amalg\overset{n)}{...}\amalg X\to X$, as expected from an ordinary constant cover. In fact, pw-connected schematic finite spaces behave well enough to replicate most constructions of the theory of finite \'etale covers of schemes, but the category (without localizing by qc-isomorphisms) is still too large to fit into the framework of Galois categories (in particular, no natural fiber functors are \textit{conservative}, see Definition \ref{definition galois category}). We think that, at least in this particular case, further topological reductions may be performed to solve this issue. 

An alternative and completely different approach, inspired by the fact that a finite schematic space can be seen to be equivalent to a certain locally ringed space and a finite sublattice of its topology, would be constructing some sort of hybrid Galois category that classifies both the algebraic and combinatorial parts of the problem.
\end{remark}

\begin{definition}\label{definition locally constant}
We say that an object $\A$ in $X_{\mathbf{Qcoh}}^{fppf}$ is \textit{locally constant }if there exists a covering $\A\to \B$ such that, if $\B\simeq \B_1\times...\times \B_n$ denotes the decomposition into well-connected components fo $\B$ (Definition \ref{definition well connected faithfully flat}), $\A\otimes_{\OO_X}\B_j$ is a constant object of degree $n_j$ in $(\B_j)_{\mathbf{Qcoh}}^{fppf}$.
\end{definition}

\begin{lemma}\label{lemma fibers are constant}
If $X$ is well-connected, $\A$ is locally constant if and only if there exists a covering $\B$ such that $f^*\A\simeq \A\otimes_{\OO_X}\B\simeq \B^n$ for some $n\geq 0$.
\begin{proof}
The <<if>> part is trivial, so let us prove the direct implication. Recall that well-connectedness implies to top-connectedness and pw-connectedness (Proposition \ref{proposition well connected is top + pw}). Using the notation of Definition \ref{definition locally constant}, assume that $k=2$, so $\mathbf{pw}(X, \B)$ has two connected components (for $k>2$ the idea is identical but the argument is a bit longer). We denote by $\B_1$ and $\B_2$ the corresponding connected components of $\B$ (Definition \ref{definition well connected faithfully flat}). We only have to prove that $n_1=n_2$. Notice that since $\B$ is faithfully flat and $X$ is pw-connected, $(X, \B)$ has non-zero stalks.

\begin{claim}
There exists some $x\in X$ such that both $\B_{1x}$ and $\B_{2x}$ are non-zero.
\begin{proof}[Proof of Claim 1]
Consider the natural morphism
\begin{align*}
\pi\colon X_\B^1\amalg X_\B^2=\mathbf{pw}(X, \B)\to (X, \B)\to X,
\end{align*}
with $X_\B^i$ the connected components of $\mathbf{pw}(X, \B)$. By definition $\B_i=\pi_{i*}\OO_{X_\B^1}$ with $\pi_i\colon X_\B^i\to X$ for $i=1, 2$.

$g$ is surjective because $\B$ has non-zero stalks. Let us prove that it is injective. Assume that either $\B_{1x}=0$ or $\B_{2x}=0$ for all $x\in X$. Note that both cannot be zero at the same time because $\B_{x}\neq 0$. In this case $\pi^{-1}(x)$ has exactly one element for all $x\in X$ by construction, so $\pi$ is injective and thus a homeomorphism. In particular $X$ would be homeomorphic to a disjoint union of non-empty topological spaces, contradicting the top-connectedness of $X$. We conclude that there must be some $x\in X$ such that both $\B_{1x}$ and $\B_{2x}$ are non-zero, which proves the claim.
\end{proof}
\end{claim}

\begin{claim}
$\B_1\otimes_{\OO_{X}}\B_2\neq 0$.
\begin{proof}[Proof of Claim 2]
Let $f_x\colon \OO_{X, x}\hookrightarrow \B_{1x}\times \B_{2x}$ denote the natural faithfully flat structure morphism and let $f_{ix}\colon \OO_{X, x}\to \B_{ix}$ (for $i=1, 2$) denote its composition with the natural projections. We have a surjective morphism
\begin{align*}
f_x^*\colon \Spec(\B_{1x})\amalg \Spec(\B_{2x})\to \Spec(\OO_{X, x})
\end{align*}
and morphisms $f_{ix}^*\colon\Spec(\B_{ix})\to \Spec(\OO_{X, x})$ with closed and open image for $i=1, 2$.

By Claim 1, there exists some $x\in X$ such that both $\B_{1x}$ and $\B_{2x}$ are non-zero. If for such an $x$, $\B_{1x}\otimes_{\OO_{X, x}}\B_{2x}=0$, taking spectra we obtain that for any $\p_i\in\Spec(\B_{ix})$ ($i=1, 2$) $f_{1x}^*(\p_1)\neq f_{2x}^*(\p_2)$, which in particular implies that
\begin{align*}
\Spec(\OO_{X, x})=\mathrm{Im}(f_x^*)\simeq \mathrm{Im}(f_{1x}^*)\amalg \mathrm{Im}(f_{2x}^*),
\end{align*}
which contradicts the hypothesis of pw-connectedness of $X$.
\end{proof}
\end{claim}
Finally, since $\B_1\otimes_{\OO_X}\B_2\neq 0$, we have isomorphisms of $\OO_X$-modules
\begin{align*}
(\B_1\otimes_{\OO_X}\B_2)^{\times n_1}\simeq \B_1^{\times n_1}\otimes_{\OO_X}\B_2\simeq  \A\otimes_{\OO_X}\B_1\otimes_{\OO_X}\B_2\simeq (\B_1\otimes_{\OO_X}\B_2)^{\times n_2},
\end{align*}
from where it follows that $n_1=n_2$, which completes the proof of the lemma.
\end{proof}
\end{lemma}

\begin{lemma}\label{lemma kernel multiplication differentials}
Let $f\colon A\to B$ be a morphism of rings with $\Omega_{A|B}=0$ and $I=\ker(B\otimes_AB\to B)$ finitely generated. Then there is a decomposition $B\otimes_AB=C\times D$ (of rings) with $I\otimes_{B\otimes_AB}C\simeq 0$, where $C$ is a $B\otimes_AB$-algebra via the natural projection. Furthermore, $C\simeq B$.
\begin{proof}
Since $\Omega_{A|B}=I/I^2=0$, by Nakayama's lemma there exists a non-zero idempotent $e\in I$ such that $I=(e)$. $C=(B\otimes_AB)/I$ and $D=(B\otimes_AB)/(1-e)$ satisfy the required conditions.
\end{proof}

\end{lemma}

Motivated by a classical result for \'etale algebras or affine schemes, but now proceeding fiberwise and checking some additional quasi-coherence conditions, we get the following result.

\begin{proposition}[Local triviality of \'etale cover sheaves]\label{proposition local triviality for \'etale sheaves}
Let $X$ be a well-connected schematic finite space. A quasi-coherent algebra $\A$ on $X$ is an \'etale cover sheaf if and only if it is finite locally constant as an object of $X_{\mathbf{Qcoh}}^{fppf}$.
\end{proposition}
\begin{proof}
Assume that there exists a quasi-coherent, finite and faithfully flat $\OO_X$-algebra $\B$ such that $\A\otimes_{\OO_X}\B\simeq \B^{\times n}$ for some $n$ (Lemma \ref{lemma fibers are constant}). We can assume that $n>0$, because otherwise $\A=0$ by faithful flatness and we would be done. Note that, since $X$ has non-zero stalks and the tensor product is defined stalkwise, this also implies that $\A_x\neq 0$ for all $x\in X$.

First, we prove that $\A$ is finite. Since, for every $x\in X$, $\OO_{X, x}\to \B_x$ is finite locally free (categories of modules are abelian, hence finite products and coproducts are isomorphic) and finiteness is a local condition, we can assume without loss of generality that it is finite and free, $\B_x\simeq \OO_{X, x}^{\oplus m}$. Since $\A_x\otimes_{\OO_{X, x}}\B_x$ is a finite $\OO_{X, x}$ module by the assumption and it is also a finite $\A_x$-module, the only possibility is that $\A_x$ is a finite $\OO_{X, x}$-module for all $x\in X$.

Now, we have to check that for any schematic point $(x, \p)$, there is an isomorphism $\A_x\otimes_{\OO_{X, x}}\overline{\kappa(x, \p)}\simeq \overline{\kappa(x, \p)}^{\times n}$ for some $n$. Indeed, we have isomorphisms:
\begin{align*}
 \overline{\kappa(x, \p)}^{\times n}\simeq \B_x^{\times n}\otimes_{\B_x} \overline{\kappa(x, \p)}\simeq (\A_x\otimes_{\OO_{X, x}}\B_x)\otimes_{\B_x} \overline{\kappa(x, \p)}\simeq \A_x\otimes_{\OO_{X, x}} \overline{\kappa(x, \p)}
\end{align*}
so $\A\in\mathbf{Qcoh}^{fet}(X)$.
\medskip

Conversely, assume that $\A$ is finite \'etale. In particular, it is locally trivial of constant degree, which we can assume positive (otherwise $\A=0$ and the statement would hold trivially). Let us proceed by induction over $n=\mathrm{deg}(\A)$. If $n=1$, then $\A\simeq \OO_X$ and there is nothing to say. In general, consider $\mathcal{I}=\ker(\A\otimes_{\OO_X}\A\to \A)$, which is a quasi-coherent sheaf of ideals (see Corollary \ref{corollary qcoh is abelian}) both as a sheaf of $\OO_X$-modules and of $\A$-modules. 

$\mathcal{I}/\mathcal{I}^2$ is the sheaf of relative differentials of $\OO_X\to \A$, which is trivial by Proposition \ref{proposition differential characterization etale} (at stalks). Since  $\A_x$ is finite and finitely presented for each $x\in X$, $\mathcal{I}_x$ is finitely generated, hence Lemma \ref{lemma kernel multiplication differentials} applies and $\mathcal{I}_x=(e_x)$ for some non-trivial idempotent $e_x\in \OO_{X, x}$ (whose image generates $\mathcal{I}_x$ as an $\A_x\otimes_{\OO_{X, x}}\A_x$-module) and
\begin{align*}
\A_x\otimes_{\OO_{X, x}}\A_x\simeq \A_x\times C_x
\end{align*}
for some ring $C_x$. Furthermore, since the natural morphism $\A_x\to \A_x\otimes_{\OO_{X, x}}\A_x$ is finite \'etale and the projection $\A_x\times C_x\to C_x$ is finite \'etale, we obtain that $C_x$ is a finite \'etale $A_x$-algebra for every $x\in X$.

Allow us to define another sheaf of ideals $\mathcal{J}$ such that $\mathcal{J}_x:=(1-e_x)$ and, for every $x\leq y$, $J_x\to J_y$ is the obvious map induced by $\mathcal{I}_x\to \mathcal{I}_y$ (note that $\A=\mathcal{I}\oplus \mathcal{J}$ as $\OO_X$-modules). Since $\mathcal{I}$ is quasi-coherent, i.e. $\mathcal{I}_y=\mathcal{I}_x\OO_{X, y}$ for all $x\leq y$, $\mathcal{J}$ is readily seen to be quasi-coherent.

Now we consider the quotient $\mathcal{C}:=(\A\otimes_{\OO_X}\A)/\mathcal{J}$, which is quasi-coherent and verifies $\mathcal{C}_x=C_x$, so it is an \'etale cover sheaf of $\OO_X$-modules by the previous discussion.

By construction, $\mathcal{C}$ (since $X$ is well-connected) has constant degree $n-1$ as a finite locally free $\A$-algebra (and $n^2-n$ as an $\OO_X$-algebra). Consider the natural qc-isomorphism $\mathbf{pw}(X, \A)\to (X, \A)$ and denote by $\mathcal{C}'$ the pullback of $\mathcal{C}$, which still has constant degree $n-1$. By the induction hypothesis, there is a faithfully flat finite algebra $\B'$ in $\mathbf{Qcoh}(\mathbf{pw}(X, \A))$ such that $\mathcal{C}'\otimes_{\OO_{\mathbf{pw}(X, \A)}}\B'\simeq \B'^{\times n-1}$. We consider its push-forward to $(X, \A)$, which is an $\A$-algebra $\B$ such that $\mathcal{C}\otimes_{\A}\B\simeq \B^{\times n-1}$.
Finally,
\begin{align*}
\A\otimes_{\OO_{X}}\B\simeq (\A\otimes_{\OO_{X}}\A)\otimes_\A\B\simeq (\A\times \mathcal{C})\otimes_\A\B\simeq \B\times \B^{\times n-1}\simeq \B^{\times n},
\end{align*}
which concludes the proof.
\end{proof}

\begin{remark} Looking at the previous proof, it is clear that $n$ is the degree of $\A$. It is possible to weaken the hypothesis of well-connectedness from Lemma \ref{lemma fibers are constant} and Proposition \ref{proposition local triviality for \'etale sheaves}, but the next section will prove that there is no loss of generality in working under this additional assumption. For a general base space $X$: if $\OO_1, ..., \OO_k$ are the well-connected components of the structure sheaf of $X$, we have morphisms $g_j\colon(X, \OO_j)\to X$ and define $\B_j:=g_{j*}g_j^*\B$ on $X$ for all $j$. A cover sheaf $\A$, after being trivialized by $\B$, is of the form $\B_1^{n_1}\times...\times \B_k^{n_k}$ for certain $n_1, ..., n_k\in\Z^+$.
\end{remark}

\begin{lemma}[Triviality of morphisms]\label{lemma triviality of morphisms}
Let $X$ be well-connected and let $f:\A\longrightarrow\B$ be a morphism in $\mathbf{Qcoh}^{fet}(X)$ such that both $\A$ and $\B$ are trivial; i.e. $\A\simeq \OO_X^{\times F}$, $\B\simeq \OO_X^{ \times E}$ for some finite discrete sets  $F, E$. Then $f$ is induced by composition with some map $\phi:F\longrightarrow E$.
\begin{proof}
We prove this at every point. Pick $x\in X$ and consider $f_x:\A_x\longrightarrow\B_x$. Since $\OO_{X, x}$ is connected (it has no non-trivial idempotents) it is an exercise of basic algebra to check that $f_x$ is indeed induced by a morphism $\phi:F\longrightarrow E$ (\cite{Lenstra}, ex. 5.11(d)). This is compatible with the restriction maps.
\end{proof}
\end{lemma}
\begin{remark}
This lemma is a formulation of the following very desirable and general property: <<morphisms between trivial covers over a connected base are determined by maps between their fibers>>. This is arguably the main property that an object worthy of being called <<connected>> in a Galois category has to verify. 
\end{remark}
\section{Stability by qc-isomorphisms}\label{section stability}

In this section we see that $\mathbf{Qcoh}^{fet}(X)$ is stable under base change by qc-isomorphisms. In combination with earlier results, this will allow us to reduce our study to spaces with sufficiently good local behavior: pw-connected schematic spaces, in which we have notions such as the degree of a finite locally free sheaf.

\medskip
Let $X$ be a schematic finite space.
\begin{lemma}\label{lemma properties qcoh sheaves on schematic spaces}
Let $\mathcal{A}\in \mathbf{Qcoh}^{alg}(X)$ be a quasi-coherent algebra, then for every affine open $U\subseteq X$ and $x\in U$ the restriction $\mathcal{A}(U)\longrightarrow\mathcal{A}_x$ is flat epimorphism of rings and $\A(U)\longrightarrow\prod_{x\in U} \A_x$ is faithfully flat. 
\begin{proof}
This follows from Propositions \ref{proposition relative spectrum}, \ref{prop restrictions are flat epic} and Corollary \ref{corollary caracterizacion afines}. 
\end{proof}
\end{lemma}

\begin{lemma}\label{lemma flat is affine local}
A quasicoherent algebra $\mathcal{A}\in \mathbf{Qcoh}^{alg}(X)$ is flat if and only if for every affine open $U\subseteq X$, the natural morphism $\OO_X(U)\longrightarrow \A(U)$ is flat. The same is true for faithful flatness. These properties can be checked on an affine open cover.
\begin{proof}
For every such $U$ we have a commutative diagram:
\begin{align*}
\xymatrix{ 
\OO_X(U)\ar[r]\ar[d]_{\varphi} & \A(U)\ar[d]^{\psi}\\
\prod_{x\in U} \OO_{X, x}\ar[r] &\prod_{x\in U} \A_x
}
\end{align*}
where $\varphi$ and $\psi$ are faithfully flat by the previous Lemma, thus the result follows, since surjectivity and exactness, hence flatness, are compatible with direct products. 

For final statement, note that if $x\in U$ with $U$ affine, we have $\A_x\simeq \A(U)\otimes_{\OO_X(U)}\OO_{X, x}$ (Equation \ref{equation fiber of qcoh in affine spaces}) and conclude by stability under base change of flat and faithfully flat morphisms. 
\end{proof}
\end{lemma}
\begin{lemma}\label{lemma finite is affine local}
A quasicoherent algebra $\mathcal{A}\in \mathbf{Qcoh}^{alg}(X)$ is finite if and only if for every affine open $U\subseteq X$, the natural morphism $\OO_X(U)\longrightarrow \A(U)$ is finite. Furthermore, it is enough to check the property for an affine open cover.
\begin{proof}
If $U$ is affine, then for any $x\in U$ we have that, $\A_x\simeq \A(U)\otimes_{\OO_X(U)}\OO_{X, x}$; so if we have a surjection $\OO_X(U)^{\oplus n}\longrightarrow \A(U)\longrightarrow 0$, every $\A_x$ is finitely generated.  Conversely, the same argument works using $\prod_{x\in U}\A_x$ and the full faithfulness of the restriction morphisms, since one has $\A(U)\otimes_{\OO_X(U)}\prod_{x\in U}\OO_{X, x}\simeq \prod_{x\in U}\A_x$.
\end{proof}
\end{lemma}
Finally we see how the \'etale property behaves under this procedure:

\begin{lemma}\label{lemma etale is affine local}
A quasicoherent algebra $\mathcal{A}\in \mathbf{Qcoh}(X)$ is \'etale if and only if for every affine open $U\subseteq X$, the natural morphism $\OO_X(U)\longrightarrow \A(U)$ is \'etale. Moreover, it is enough to check this property for an affine open cover.
\begin{proof}

Assume that $\A(U)$ is \'etale for some affine open $U$. Since $U$ is affine, by Lemma \ref{lemma properties qcoh sheaves on schematic spaces} and Proposition \ref{proposition flat epic are local iso} we have that, for $r_x\colon \A(U)\to \A_x$ and every prime $\p\subset\A_x$,
\begin{align*}
\A(U)_{r_x^*\p}\simeq (\A_x)_\p.
\end{align*}
 The result follows straightforwardly from Lemmas \ref{lemma characterization etale algebras}, \ref{lemma flat is affine local} and \ref{lemma finite is affine local}.

For the converse, since $\OO(U)\longrightarrow\prod_{x\in U}\OO_{X, x}$ is faithfully flat, it induces a surjective morphism between spectra. In particular any $\q\in\OO(U)$ is the image of a prime ideal $\p\subseteq \OO_{X, x}$ for some $x\in U$. Applying the previous argument and using the hypothesis we conclude.
\end{proof}
\end{lemma}

\begin{theorem}\label{theorem site stability by qc-iso}
Let $f\colon X\to Y$ is a qc-isomorphism of schematic finite spaces. The natural pair $(f^*, f_*)\colon \mathbf{Qcoh}^{alg}(X)\overset{\sim}{\to}\mathbf{Qcoh}^{alg}(Y)$ induces an adjoint equivalence of sites given by the functor
$$f_*\colon X_{\mathbf{Qcoh}}^{fppf}\overset{\sim}{\to}Y_{\mathbf{Qcoh}}^{fppf}$$
with adjoint $f^*$.
\begin{proof}
We already know that $(f^*, f_*)$ is an adjoint equivalence of categories. Furthermore, these functors preserve coverings: if $\A\in X_{\mathbf{Qcoh}}^{fppf}$ and $\A\to \B$ is a covering (finite and faithfully flat), since qc-isomorphisms are affine, $f_*\A\to f_*\B$ is a covering of $f_*\A$ on $Y_{\mathbf{Qcoh}}^{fppf}$ by Lemmas \ref{lemma flat is affine local} and \ref{lemma finite is affine local}. Conversely, let $\A\to \B$ be a covering in $Y_{\mathbf{Qcoh}}^{fppf}$. Since $f^*\A\simeq f^{-1}\A$ and $f^*\B\simeq f^{-1}\B$ due to the qc-isomorphism assumption, $f^*\A\to f^*\B$ is a covering in $X_{\mathbf{Qcoh}}^{fppf}$. This implies the claim.
\end{proof}
\end{theorem}
\begin{remark}
We are regarding $(f^*, f_*)$ as an equivalence between two sites whose underlying categories are made of quasi-coherent sheaves. Note that this equivalence of sites refers to a proper equivalence of categories and coverages, which is far stronger than them having equivalent sheaf topoi.
\end{remark}

\begin{corollary}\label{theorem stability by qc-iso}
Let $f:X\longrightarrow Y$ be a qc-isomorphism of schematic finite spaces. The pair
\begin{align*}
(f_*, f^*)\colon \mathbf{Qcoh}^{fet}(X)&\longrightarrow \mathbf{Qcoh}^{fet}(Y)\\
\A&\longrightarrow f_*\A
\end{align*}
is an equivalence of categories.
\begin{proof}
It follows from Theorem \ref{theorem site stability by qc-iso} and Proposition \ref{proposition local triviality for \'etale sheaves}. A direct proof follows from Lemma \ref{lemma etale is affine local}.
\end{proof}
\end{corollary}

This corollary together with the fact that qc-isomorphisms induce isomorphisms between geometric points (Theorem \ref{theorem functorial qc-isomorphisms}) implies that our Galois categories are stable under qc-isomorphisms. In particular we can use the results on connectedness to slightly improve our main result.

\section{Finite \'Etale Covers as a Galois Category}

Let $S$ be a scheme and let $X$ be a certain finite model. Fix a geometric point $\overline{s}\in S^\bullet(\Omega)$ with image $\overline{x}\in X^\bullet(\Omega)$. The category of finite \'etale covers of $S$ is equipped with a natural functor called <<fiber functor>> (see subsection \ref{subsection fiber functors}):
\begin{align*}
\textbf{Fib}_{\overline{s}}:\textbf{Fet}_S&\longrightarrow \mathbf{Set}\\
(T\rightarrow S)&\longrightarrow |\Spec(\Omega)\times_ST|,
\end{align*}
which composed with the equivalence $\mathbf{Qcoh}^{fet}(X)^{\mathrm{op}}\simeq \textbf{Fet}_S$ gives us an analogous <<fiber functor>> for $X$ and the fixed geometric point $\overline{x}$. If $X$ is now a completely general schematic finite space and $\overline{x}\in X^\bullet(\Omega)$, we can define
\begin{align}
\mathrm{Fib}_{\overline{x}}:\mathbf{Qcoh}^{fet}(X)^{\mathrm{op}}&\longrightarrow \mathbf{Set}\\
\nonumber \A&\longrightarrow |\Spec(\Omega\otimes_{\OO_{X, x}}\A_x)|.
\end{align}

\begin{remark}
\textit{A priori}, this functor is the following composition (see equation \ref{diagrama functores fibra})
\begin{align}
\mathbf{Qcoh}^{fet}(X)^{\mathrm{op}}\overset{\Phi}{\longrightarrow}({\mathbf{SchFin}_{/X}})_{qc}\overset{\text{F}_{\overline{x}}}{\longrightarrow}\mathbf{SchFin}_{qc}\overset{\mathfrak{Spec}}{\longrightarrow}\textbf{Schemes}\overset{|-|}{\longrightarrow}\mathbf{Set}
\end{align}
where $\Phi$ is the analogue of the relative spectrum for schematic spaces and $|-|$ is taking the underlying set. Thus $\mathrm{Fib}_{\overline{x}}$ sends $\A\in\mathbf{Qcoh}^{fet}(X)^{\mathrm{op}}$ to
\begin{align*}
\mathcal{A}\overset{\Phi}{\mapsto}\big((X, \A)\rightarrow (X, \OO_X)\big)\overset{\text{F}_{\overline{x}}}{\mapsto}(\star, \Omega\otimes_{\OO_{X, x}}\A_x)\overset{|\mathfrak{Spec}|}{\mapsto}|\Spec(\Omega\otimes_{\OO_{X, x}}\A_x)|
\end{align*}

Note that since $X$ is not necessarily the finite model of a scheme, $\mathfrak{Spec}$ takes values in the larger category of (locally) ringed spaces, but since the composition $\text{F}_{\overline{x}}\circ \Phi$ takes values in the category of affine schematic finite spaces, their composition with $\mathfrak{Spec}$ does take values in $\textbf{Schemes}$.
\end{remark}

The following proposition proves that $\mathrm{Fib}_{\overline{x}}$ coincides exactly with the definition one would expect: the underlying set of the fibered product along the geometric point $\overline{x}$.

\begin{proposition}\label{proposition fibers of etale}
The functor $\mathrm{Fib}_{\overline{x}}$ defined above coincides with the composition $|-|\circ \mathrm{F}_{\overline{x}}\circ\Phi$, where $\Phi\colon \mathbf{Qcoh}^{fet}(X)^{\mathrm{op}}\to \mathbf{SchFin}$ is the relative spectrum functor of Proposition \ref{proposition relative spectrum}, $\mathrm{F}_{\overline{x}}$ is the geometric fiber functor of Definition \ref{definition geometric fiber} and $|-|$ is the forgetful functor to $\mathbf{Set}$. In other words, 
\begin{align*}
\mathrm{Fib}_{\overline{x}}(\A)=|\mathbf{pw}((\star, \Omega)\times_X(X, \A))|.
\end{align*}
\begin{proof}
Proposition \ref{proposition etale algebras over fields} implies that $(\star, \Omega)\times_X(X, \A)=(\star, \Omega\otimes_{\OO_{X, x}}\A_x)=(\star, \prod_I\Omega)$ for some finite set $I$. Now, it is obvious that 
\begin{align*}
|\mathbf{pw}(\star, \prod_I\Omega)|=|\coprod_I(\star, \Omega)|=I=|\Spec(\prod_I\Omega)|=|\Spec(\Omega\otimes_{\OO_{X, x}}\A_x)|,
\end{align*}
which proves the claim.
\end{proof}
\end{proposition}
\begin{remark}
Notice that we do not consider the functors of the proposition in $\mathbf{SchFin}_{qc}$, since the underlying set of different representatives of the fiber may change, i.e. $|-|\colon \mathbf{SchFin}^{pw}\to\mathbf{Set}$ does not factor through $\mathbf{SchFin}_{qc}^{pw}$. Once again, we think that there exists a sufficiently small subcategory in which this is possible, at least for affine morphisms over a fixed base ($(\star, \Omega)$ in this case). The same questions of Remark \ref{remark remove localization} arise.
\end{remark}

Let us see that the pair $(\mathbf{Qcoh}^{fet}(X), \mathrm{Fib}_{\overline{x}})$ is a Galois category in the sense of Grothendieck (see \cite{SGAI}, \cite{Lenstra}). Let us recall the general definition of a Galois Category (some axioms may be weakened at the cost of some properties, but we shall not be concerned with those versions for now).

\begin{definition}[Galois Category]\label{definition galois category}
Let $\mathfrak{C}$ be a category and $F:\mathfrak{C}\longrightarrow \mathbf{Set}_f$ a covariant functor to the category of finite sets. We say that the pair $(\mathfrak{C}, F)$ is a \textit{Galois category} with fundamental functor $F$ if:
\begin{enumerate}
\item $\mathfrak{C}$ has a terminal object and finite fibered products.
\item $\mathfrak{C}$ has finite sums, in particular an initial object, and the quotients by a finite group of automorphisms exist for every object of $\mathfrak{C}$.
\item Any morphism $u$ in $\mathfrak{C}$ can be written as $u=u'\circ u''$ where $u''$ is an epimorphism and $u'$ a monomorphism. Additionally, any monomorphism $u:X\longrightarrow Y$ in $\mathfrak{C}$ is an isomorphism of $X$ with a direct summand of $Y$.
\item $F$ preserves terminal objects, epimorphisms and commutes with fibered products, finite sums and quotients by finite groups of automorphisms.
\item $F$ is conservative, i.e. if $u$ is a morphism such that $F(u)$ is an isomorphism, then $u$ is an isomorphism. 
\end{enumerate}
\end{definition}
\begin{remark}[Standard properties of Galois categories]
The fundamental group of a Galois category $(\mathfrak{C}, F)$ is $\pi_1^\mathfrak{C}:=\text{Aut}_{[\mathfrak{C}, \mathbf{Set}_f]}F$.  An object in a \textit{Galois} category $X\in \mathrm{Ob}(\mathfrak{C})$ is called Galois if it is connected and its quotient under the full group of automorphisms is isomorphic to the final object. Moreover, $F$ is a (strictly) \textit{pro-representable} functor, which means that there is an isomorphism $F\simeq\varinjlim\Hom_{\mathfrak{C}}(-, P_i)$, where the right hand side is a directed colimit running over all Galois objects, with surjective transition maps between them. More precisely, $F$ lifts to a $\mathbf{Set}$-valued functor from the category of pro-objects  of $\mathfrak{C}$, $F\colon\mathrm{pro}\text{-}\mathfrak{C}\to\mathrm{pro}\text{-}\mathbf{Set}_f\overset{\lim}{\to}\mathbf{Set}$, which is representable by $(P_i)$, the pro-object defined by all Galois objects, and a \textit{universal point} $(\alpha_i)\in\varprojlim_i F(P_i)$. The isomorphism $F\simeq \Hom_{\mathrm{pro}\text{-}\mathfrak{C}}(-, P_i)_{|\mathfrak{C}}$ is given by evaluation at such universal point. By Yoneda's Lemma, the fundamental group is now isomorphic to $\mathrm{Aut}_{\mathrm{pro}\text{-}\mathfrak{C}}(P_i)$, which is a profinite group. This profinite group can be realized as an actual limit $\varprojlim \mathrm{Aut}_{\mathfrak{C}}(P_i)$ in the category of topological groups. Finally, $F$ naturally enriches to an equivalence $F\colon\mathfrak{C}\overset{\sim}{\to} \pi_1^{\mathfrak{C}}\text{-}\mathbf{Set}_f$, where the target denotes the category of finite sets with a \textit{continuous} action of $\pi_1^{\mathfrak{C}}$. All these claims are standard properties of all Galois categories, as originally introduced in \cite[Expos\'e V]{SGAI}.
\end{remark}

We have already made all the necessary preparations to prove our main result.
\begin{theorem}\label{theorem main}
Let $X$ be a connected schematic finite space and let $\overline{x}\in X^\bullet(\Omega)$ be a geometric point. The pair $(\mathbf{Qcoh}^{fet}(X), \mathrm{Fib}_{\overline{x}})$ is a Galois Category. Furthermore, if $X$ is the finite model of a scheme $S$ and $\overline{s}\in S^\bullet(\Omega)$ is the corresponding geometric point (Proposition \ref{proposition points scheme are schematic points of the model}), then we have an isomorphism of profinite groups
\begin{equation}
\pi_1^{et}(S, \overline{s})\simeq \pi_1^{et}(X, \overline{x})
\end{equation}
where $\pi_1^{et}(X, \overline{x}):=\mathrm{Aut}_{[\mathbf{Qcoh}^{fet}(X)^{\mathrm{op}}, \mathbf{Set}_f]}(\mathrm{Fib}_{\overline{x}})$. 
\begin{proof}
We can assume that $X$ is well-connected due to Theorems \ref{theorem pw connectification},  \ref{theorem functorial qc-isomorphisms} and Corollary \ref{theorem stability by qc-iso}. For such an $X$ we have a well defined notion of degree of $\A\in\mathbf{Qcoh}^{fet}(X)$ (identified with an integer). Subsections \ref{section axioms 1, 2}, \ref{section axiom 3} and \ref{section axioms 4, 5} below prove that, indeed, the axioms of Definition \ref{definition galois category} are satisfied in this case. The arguments we employ to prove that the axioms hold consist of showing that we can reduce it to checks at stalks---where we ultimately use a contravariant version of the ideas used in \cite{Lenstra} and \cite{Szamuely}---, showcasing the techniques required to work with schematic finite spaces.
\end{proof}
\end{theorem}

\subsection{Verification of the first and second axioms:}\label{section axioms 1, 2}
\begin{proposition}
$\mathbf{Qcoh}^{fet}(X)$ has an initial object and finite direct sums.
\begin{proof}
The initial object is the structure sheaf $\OO_X$. Our candidate for direct sum over a third object is the tensor product of algebras. Let $\mathcal{A}, \mathcal{B}, \mathcal{C}\in   \mathbf{Qcoh}^{fet}(X)$ be such that $\mathcal{A}$ and $\mathcal{B}$ are also $\mathcal{C}$-algebras. We have to check that $\A\otimes_{\mathcal{C}}\B\in\mathbf{Qcoh}^{fet}(X)$. This can be checked at the level of stalks, where it holds true because all the properties of ring homomorphisms that are involved (finiteness, flatness, \'etaleness) are well known to be stable under arbitrary base changes and verify the <<cancellation property>>: if $A\to B$ verifies a property $\mathbf{P}$ and $A\to B\to C$ verifies $\mathbf{P}$, then $B\to C$ verifies $\mathbf{P}$, i.e. every morphism in
\begin{align*}
\xymatrixrowsep{0.1in}\xymatrix{ 
& & \A \ar[rd] & \\
\OO_X \ar[urr]\ar[r]\ar[drr] & \mathcal{C}\ar[ru]\ar[dr] & & \A\otimes_{\mathcal{C}}\B\\
& & \B \ar[ru] &
}
\end{align*}
is \'etale at stalks.
\end{proof}
\end{proposition}
\begin{corollary}
$\mathbf{Qcoh}^{fet}(X)^{\mathrm{op}}$ verifies axiom 1 of Definition \ref{definition galois category}.
\end{corollary}
 
\begin{proposition}
$\mathbf{Qcoh}^{fet}(X)$ has finite direct products, in particular a terminal object.
\begin{proof}
Given  two finite \'etale cover sheaves $\mathcal{A}$ and $\mathcal{B}$, it is easy to check that their direct product $\mathcal{A}\times\mathcal{B}$ (with natural projections $\pi_\A$ and $\pi_\B$) is an \'etale cover sheaf and satisfies the requirements. Explicitly, the terminal object is the zero sheaf.
\end{proof}
\end{proposition}
\begin{proposition}
$\mathbf{Qcoh}^{fet}(X)$ has contravariant categorical quotients by (finite) subgroups of automorphisms.
\begin{proof}
Assume without loss of generality that $X$ is pw-connected. Consider the group of automorphisms $\mathrm{Aut}_{\OO_X}(\mathcal{A})\equiv \mathrm{Aut}(\OO_X\rightarrow\mathcal{A})$ and a finite subgroup $G\subseteq \mathrm{Aut}_{\OO_X}(\mathcal{A})$. We can define:
\begin{align*}
Q(\mathcal{A}, G):\mathbf{Qcoh}^{fet}(X)^{\mathrm{op}}&\longrightarrow \mathbf{Set}\\
\mathcal{B}&\longmapsto \Hom_{\OO_X}(\mathcal{B}, \mathcal{A})^G
\end{align*}
where $\Hom_{\OO_X}(\mathcal{B}, \mathcal{A})^G:=\{f\in\Hom_{\OO_X}(\B, \A):\phi\circ f=f\text{ }\forall\phi\in G\}$ is the subset of $G$-invariant morphisms. The categorical quotient is by definition a representing object for this functor. In this case, it will be the subsheaf $\mathcal{A}^G$ of invariant elements under the action of $G$, defined at stalks as
\begin{align*}
\mathcal{A}^G_x:=\{a\in\mathcal{A}_x:\phi_x(a)=a\text{ for every }\phi\in G\}\subseteq \mathcal{A}_x
\end{align*}
for every $x\in X$; where $\phi_x\colon \A_x\to\A_x$ is the stalk of $\phi$ at $x$.

It is indeed a sub-$\OO_X$-module, because every element of $\OO_X$ is $G$-invariant by definition when thought of as an element of $\mathcal{A}$ (for a certain open set). Since automorphisms act trivially on the structure sheaf, for every $x\leq y\in X$ one has
\begin{align*}
\A_y^G\simeq (\mathcal{A}_x\otimes_{\OO_{X, x}}\OO_{X, y})^G=\mathcal{A}_x^G\otimes_{\OO_{X, x}}\OO_{X, y}\end{align*}
for the naturally induced actions, and thus $\A^G\in\mathbf{Qcoh}^{alg}(X)$. 

Now we prove that it is an \'etale cover sheaf. We can argue on each well-connected component, so we may assume that $X$ itself is well-connected. Since a finite direct product of \'etale morphisms is \'etale, we may assume that $\A$ is connected as well. By virtue of $\A$ being an \'etale cover sheaf, Proposition \ref{proposition local triviality for \'etale sheaves} implies that there is a finite and faithfully flat $\OO_X$-algebra $\CC$ s.t. $\A\otimes_{\OO_X}\CC\simeq \CC^{\times n}$ (with $n=\mathrm{deg}(\A)$). The action of $G$ on $\A$ induces a tensor product action on $\A\otimes_{\OO_X}\CC$ (trivial on $\CC$), and thus---via the trivialization---an action on $\CC^{\times n}\simeq \CC\otimes_{k}(k^{\times n})$ (we tensor stalkwise by a finite free algebra!) by permutations of the factors. We have
\begin{align*}
\A^G\otimes_{\OO_X}\CC\simeq (\A\otimes_{\OO_X}\CC)^G\simeq (\CC^{\times n})^G\simeq \CC^{\times m}
\end{align*}
for some $m<n$. We conclude by the converse of Proposition \ref{proposition local triviality for \'etale sheaves}.

Finally, $\mathcal{A}^G$ represents the functor $Q(\A, G)$ because every $G$-invariant morphism $f\colon \B\to \A$ has its image contained in $\A^G$, so it factors as $\B\to \A^G\hookrightarrow \A$. Conversely, we compose any map $g\colon \B\to \A^G$ with the natural inclusion. Both correspondences are mutually inverse.
\end{proof}
\end{proposition}

\begin{remark}
 $\mathrm{Aut}_{\OO_{X}}(\A)\subseteq \prod_{x\in X}\mathrm{Aut}_{\OO_{X, x}}\A_x$ is automatically seen to be a finite group, since each one of the components of the product is well known to be finite.
\end{remark}

\begin{corollary}
$\mathbf{Qcoh}^{fet}(X)^{\mathrm{op}}$ verifies axiom 2 of Definition \ref{definition galois category}.
\end{corollary}

\subsection{Verification of the third axiom:}\label{section axiom 3}
The following statement is a standard consequence of Nakayama's lemma:
\begin{lemma}\label{lemma support and zeroes}
Let $M$ be a finite $A$-module, $C\subseteq \Spec(A)$ a closed subset and $I=I(C)$ the radical ideal generated by $C$. Then the $A$-module $M/IM$ verifies that:
$$\mathrm{Supp}(M/IM)=\mathrm{Supp}(M)\cap V(I)$$
where $V(I)=C$. In particular $(M/IM)_\p=0$ for every $\p\not\in C$.
\end{lemma}
\begin{lemma}\label{lemma support locally free rank 1}
Let $A$ be a ring and let $C$ be and open and closed subset of $\Spec(A)$ such that $I=I(C)$. Then for every every $\p\in C$ we have
\begin{align*}
(A/I)_\p\simeq A_\p.
\end{align*}
\begin{proof}
Since $C$ is open and closed, $\Spec(A)=C\amalg D$ with $D=V(J)=\Spec(A/J)$ for some ideal $J$. In particular we have a decomposition $A\simeq A/I\times A/J$. The result follows from Lemma \ref{lemma support and zeroes}.
\end{proof}
\end{lemma}

Furthermore, since \'etale algebras are finite and locally free, we have the following result.
\begin{lemma}\label{lemma support clopen}
If $f\colon A\to B$ is an \'etale morphism of rings, then $\mathrm{Supp}(B)$ (as an $A$-module) is open and closed.
\begin{proof}
It follows from $\Spec(A)-\mathrm{Supp}(B)=\{\p\in\Spec(A): \mathrm{deg}(B)(\p)=0\}$ and continuity of the degree map.
\end{proof}
\end{lemma}
Since injective (resp. surjective) morphisms of sheaves of rings (ring data) are monomorphisms (resp. epimorphisms) and $\mathbf{Qcoh}^{fet}(X)$ is a subcategory of $\mathbf{Ring}\text{-}\mathbf{data}_X$, we have:
\begin{lemma}\label{lemma injective is monic}
Let $f:\A\longrightarrow\B$ be a morphism in $\mathbf{Qcoh}^{fet}(X)$. If $f$ is injective (resp. surjective), then it is a monomorphism (resp. an epimorphism).
\end{lemma}

If $f\colon A\to B$ is a ring homomorphism, it is clear that $\ker(f)=\mathrm{Ann}(B)$ (where $\mathrm{Ann}(B)$ denotes the annihilator ideal of the $A$-module $B$. Note that $\mathrm{Ann}(B)=I(\mathrm{Supp}(B))$ because $B$ is  a finite $A$-module and $\ker(f)$ verifies that $\ker(f)_x=\mathrm{Ann}(\B_x)$ for all $x\in X$. Finally, if $\A$ and $\B$ are quasi-coherent, by Corollary \ref{corollary qcoh is abelian} we obtain that $\ker(f)$ is a quasi-coherent ideal.

\begin{proposition}\label{proposition factorization}
Any $f:\mathcal{A}\longrightarrow\mathcal{B}$ in $\mathbf{Qcoh}^{fet}(X)$ can be written as $f=h\circ g$ with $g$ an epimorphism and $h$ a monomorphism. 
\begin{proof}
We have the ordinary factorization $\A\overset{g}{\to}\A/\ker(f)\overset{h}{\to}\B$. Since $\ker(f)$ is quasi-coherent by Corollary \ref{corollary qcoh is abelian}, it follows, from the same corollary, that $\A/\ker(f)$ is a quasi-coherent algebra. $g$ is an epimorphism because it is a quotient map. $h$ is injective at stalks by the first isomorphism Theorem, thus a monomorphism by Lemma \ref{lemma injective is monic}.

We see that $\A/\ker(f)$ is an \'etale $\OO_X$-algebra at stalks, for which we use Proposition \ref{proposition punctually etale}. First, notice that by the cancellation property for \'etale morphisms, $\B_x$ is an \'etale $\A_x$-algebra for all $x\in X$, i.e. $\B\in\mathbf{Qcoh}^{fet}(X, \A)$. By the previous discussion $(\A/\ker(f))_x=\A_x/\mathrm{Ann}(\B_x)$. Finiteness holds because every quotient of a finite module is finite. Flatness can be checked after localizing at every prime, so we win by Lemmas \ref{lemma support clopen} and \ref{lemma support locally free rank 1}. Pointwise-\'etaleness follows by the same argument. We conclude that $\A/\ker(f)$ is an \'etale cover sheaf of $\OO_X$-algebras. Note that $\B_x$ is a finite locally free $(\A/\ker(f))_x$-algebra of strictly positive degree.
\end{proof}
\end{proposition}

\subsubsection*{More on monomorphisms and epimorphisms.}
To complete the verification of the third axiom, we need a complete characterization of monic and epic morphisms in $\mathbf{Qcoh}^{fet}(X)$.
\begin{lemma}\label{epic morph are punctual}
A morphism $f:\A\longrightarrow\mathcal{B}$ in $\mathbf{Qcoh}^{fet}(X)$ is an epimorphism if and only if for every $x\in X$, $f_x:\A_x\longrightarrow\mathcal{B}_x$ is an epimorphism in the category of (flat, finite) \'etale algebras over $\OO_{X, x}$.
\begin{proof}
If all the local rings verify the condition, it is clear that the global morphism also verifies it. 
For the converse, it suffices to see that for any $x\in X$ and any morphism $g:\B_x\longrightarrow R$ (with $R$ finite, flat and \'etale over $\OO_{X, x}$), there exists a morphism of sheaves $f:\B\longrightarrow\mathcal{R}$ such that $\mathcal{R}_x=R$, $f_x=g$ and $\mathcal{R}\in\mathbf{Qcoh}^{fet}(X)$.

Notice that every $R$ determines a sheaf of algebras $\mathcal{R}^x$ on $U_x$ defined as $\mathcal{R}^x_y:=R\otimes_{\OO_{X, x}}\OO_{X, y}$ equipped with a natural morphism $i^*\B\longrightarrow \mathcal{R}^x$ (where $i:U_x\hookrightarrow X$ is the natural inclusion). The extension theorem for quasi-coherent sheaves on schematic spaces (Theorem \ref{theorem extension of quasicoherent sheaves}) gives us
\begin{align*}
i_*i^*\B\equiv B_{U_x}\longrightarrow \mathcal{R}:=i_*\mathcal{R}^x;
\end{align*}
which composed with the natural morphism $\B\longrightarrow i_*i^*\B$ gives us the desired map. It remains to check that $\mathcal{R}\in\mathbf{Qcoh}^{fet}(X)$, which is easy because for every $y\geq x$,
\begin{align*}
\mathcal{R}_y=R\otimes_{\OO_{X, x}}\OO_{X, y}
\end{align*}
and for every other $y$ we have 
\begin{align*}
\mathcal{R}_{y}=\mathcal{R}^x(U_x\cap U_y)=R\otimes_{\OO_{X, x}}\OO_X(U_x\cap U_y);
\end{align*}
and all the conditions are stable after a base change. 
\end{proof}
\end{lemma}

From this Lemma and the classical theory of \'etale $\OO_{X, x}$-algebras (morphisms between \'etale algebras are \'etale and epimorphisms of rings of finite presentation are surjective), we have:
\begin{proposition}\label{prop charac epic}
A morphism $f:\mathcal{A}\longrightarrow\mathcal{B}$ in $\mathbf{Qcoh}^{fet}(X)$ is an epimorphism if and only if $f_x$ is surjective for every $x\in X$, i.e. if and only if $f$ is surjective.
\end{proposition}

We proceed now to characterize monomorphisms.
\begin{proposition}\label{prop charac monic}
$f:\mathcal{A}\longrightarrow\mathcal{B}$ in $\mathbf{Qcoh}^{fet}(X)$ is a monomorphism if and only if it is injective.
\begin{proof}One direction is Lemma \ref{lemma injective is monic}. For the converse: since $\mathrm{Supp}(\B_x)$ (as an $\A_x$-module) is open and closed for every $x\in X$, its complement in $\Spec(\A_x)$ defines an ideal $\ker(f_x)^C$ for all $x\in X$. This defines a quasi-coherent sheaf of ideals $\ker(f)^C$ such that
\begin{align*} 
\mathcal{A}\simeq\mathcal{\mathcal{A}}/\ker(f)\times  \mathcal{A}/\ker(f)^C.
\end{align*}

Let $f$ be a monomorphism. Denote $\A_1=\A/\ker(f)$, $\A_0=\A/\ker(f)^C$ and consider the natural projections $\pi_1, \pi_2\colon \A_0\times \A_0\to \A_0$ and
\begin{align*}
\A_0\times \A_0\times \A_1\overset{\pi_1}{\underset{\pi_2}{\rightrightarrows}} \A_0\times\A_1\simeq\mathcal{A}\overset{f}{\longrightarrow} \mathcal{B}.
\end{align*}
By definition of $\A_0$ (or simply the factorization of Proposition \ref{proposition factorization}), we have $f\circ\pi_1=f\circ \pi_2$; and since $f$ is a monomorphism, $\pi_1=\pi_2$, thus $\A_0\simeq 0$ (the ring with one element). Now it is clear that $\ker(f)=0$, hence $f$ is injective (see proof of Proposition \ref{proposition factorization} as well). 
\end{proof}
\end{proposition}

\begin{proposition}
Any epimorphism $f:\mathcal{A}\longrightarrow\mathcal{B}$ in $\mathbf{Qcoh}^{fet}(X)$ induces $\mathcal{C}\overset{\sim}{\longrightarrow}\mathcal{B}$ such that $\mathcal{A}\simeq \mathcal{C}\times \mathcal{D}$.
\begin{proof}
The same decomposition of Proposition \ref{prop charac monic} with $\mathcal{C}=\A_1$ and $\mathcal{D}=\A_0$ satisfies the conditions. Indeed,  $f$ is surjective by Proposition \ref{prop charac epic}, thus $\mathcal{C}\to\B$ is surjective; and also a monomorphism by Proposition \ref{proposition factorization}, hence injective by Proposition \ref{prop charac monic} and therefore it is  an isomorphism.
\end{proof}
\end{proposition}

\begin{corollary}
$\mathbf{Qcoh}^{fet}(X)^{\mathrm{op}}$ verifies axiom 3 of Definition \ref{definition galois category}.
\end{corollary}

\subsection{Verification of the fourth and fifth axioms:}\label{section axioms 4, 5}

\begin{corollary}
$(\mathbf{Qcoh}^{fet}(X)^{\mathrm{op}}, \mathrm{Fib}_{\overline{x}})$ verifies axiom 4 of Definition \ref{definition galois category}.
\begin{proof}

$\mathrm{Fib}_{\overline{x}}$ clearly sends the initial object $\OO_X$ (the final object of the opposite category) to the set consisting of one point (the final object of $\mathbf{Set}$).

If $\A\longrightarrow\mathcal{B}$ is a monomorphism in $\mathbf{Qcoh}^{fet}(X)$, then it is injective at every point and, consequently, the induced map between their localizations is injective; a property which is preserved after tensoring by a field (flatness) and thus induces a surjection between the spectra of their local rings that induces a surjection between the fibers.

Similarly, the tensor product of \'etale cover sheaves commutes with localization, and thus is mapped to the fibered product of schemes, whose underlying set coincides with the fibered product of sets because we are tensoring by a field. The direct sum of cover sheaves is trivially sent to the disjoint union of their respective fibers.

For the last part, just note that given $\A$ in $\mathbf{Qcoh}^{fet}(X)$ and a finite subgroup of automorphisms 
$G\subseteq \mathrm{Aut}_{\OO_{X}}\A$, the antiequivalence of categories (for every $x\in X$) between $\OO_{X, x}$-algebras and affine $\Spec(\OO_{X, x})$-schemes yields $$\Spec((\A^G)_x)=\Spec((\A_x)^G)=\Spec(\A_x)/G,$$ 
where taking invariants commutes with localization by any multiplicative subset (because automorphisms leave the base ring invariant) and the last quotient is taken with respect to the natural action induced on the spectrum ($\mathrm{Aut}_{\OO_{X, x}}\A_x=\mathrm{Aut}_{\Spec(\OO_{X, x})}\Spec(\A_x)$ for all $x\in X$).

The conclusion follows because the action of $G$ on $\Omega$ is trivial, so for any $\overline{x}\in X^\bullet(\Omega)$ we have $\Fib_{\overline{x}}(\A^G)=|\Spec(\Omega\otimes_{\OO_{X, x}}\A_x^G)|=|\Spec((\Omega\otimes_{\OO_{X, x}}\A_x)^G)|=\Fib_{\overline{x}}(\A)/G$.
\end{proof}
\end{corollary}

\begin{lemma}\label{lemma last lemma}
Let $X$ be well-connected and let $f:\A\longrightarrow\B$ be an injective morphism in $\mathbf{Qcoh}^{fet}(X)$. If $\mathrm{deg}(\A)=\mathrm{deg}(\B)$ as $\OO_X$-modules, then $f$ is an isomorphism.
\begin{proof}
Take two non-zero finite and locally free $\OO_X$-modules $\mathcal{C}$ and $\mathcal{D}$ trivializing $\A$ and $\B$ as in Proposition \ref{proposition local triviality for \'etale sheaves}; their tensor product as $\OO_X-$modules, denoted $\mathcal{G}$, trivializes both simultaneously. Notice that $\mathcal{G}$ is a finite faithfully flat $\OO_X-$module because both of its components are faithfully flat, so it is a covering trivializing both $\A$ and $\B$ in $X_{\mathbf{Qcoh}}^{fppf}$. Now, Lemma \ref{lemma triviality of morphisms} yields that the base change $\A\otimes_{\OO_X}\mathcal{G}\longrightarrow\B\otimes_{\OO_X}\mathcal{G}$, which is \textit{a priori} injective due to faithful flatness, is an isomorphism. Once again by faithful flatness, we conclude that $f$ is an isomorphism. 
\end{proof}
\end{lemma}

\begin{corollary}
If $X$ is well-connected, $(\mathbf{Qcoh}^{fet}(X)^{\mathrm{op}}, \mathrm{Fib}_{\overline{x}})$ verifies axiom 5 of Definition \ref{definition galois category}.
\begin{proof}
Consider $u:\A\longrightarrow\mathcal{B}$ such that $\mathrm{Fib}_{\overline{x}}(u)$ is an isomorphism. We have the usual factorization $\A=\A_0\times\A_1\longrightarrow\A_1\longrightarrow \B$, with the first morphism surjective and the second one injective. Since $\Fib_{\overline{x}}$ sends direct products in $\mathbf{Qcoh}^{fet}(X)$ to disjoint unions, the hypothesis implies that $\mathrm{Fib}_{\overline{x}}(\A_1)\longrightarrow \mathrm{Fib}_{\overline{x}}(\A_0)\amalg \mathrm{Fib}_{\overline{x}}(\A_1)$ is surjective, and thus $\mathrm{Fib}_{\overline{x}}(\A_0)=\emptyset$. Since the degree of $\A$ as an $\OO_X-$module is constant, this implies that $\A_0=0$ and in particular $u:\A=\A_1\longrightarrow\B$ is an injective morphism.  

Since $\mathrm{Fib}_{\overline{x}}(u)$ is an isomorphism, $\mathrm{Fib}_{\overline{x}}(\A)$ and $\mathrm{Fib}_{\overline{x}}(\B)$ have the same number of elements, i.e. $\mathrm{deg}(\A)=\mathrm{deg}(\B)$ as $\OO_X$-modules. Lemma \ref{lemma last lemma} concludes the proof. 
\end{proof}
\end{corollary}

\section{Further work and examples}

The theory presented in this paper is only concerned with the existence of the fundamental group, and thus it is not practical to produce significant examples, these will come in future work. However, let us outline some elementary ideas that put this theory into perspective.

To begin with, we claim that the \'etale fundamental group of schematic spaces (not necessarily schemes) verifies most of the standard properties that one would expect for  the \'etale fundamental group. For most of them we will need to introduce additional notions regarding finite schematic spaces, ranging from a good <<schematic dimension theory>> to an adequate notion of proper morphisms, which is a task of independent interest that will be treated in future works. Even without requiring these concepts, as mentioned in the introduction, in future papers we will prove several versions of the Seifert-Van Kampen Theorem that will allow us to compute practical examples with rather elementary technology. The simplest one of them can be stated as follows:
\begin{theorem}[<<Internal>> Seifert-Van Kampen for topologically irreducible spaces]
Let $X$ be a connected schematic finite space that is topologically irreducible, i.e. $X=C_x$ for some $x\in X$. Let $\overline{x}$ be a geometric point with image $x$ and, for every $y\leq x$, let $\overline{x}_y$ be the geometric point of $\mathrm{Spec}(\OO_{X, y})$ defined by $\overline{x}$ (i.e. if $\overline{x}$ is defined by the ideal $\p\subseteq \OO_{X, x}$ and a field extension $\kappa(x, \p)\to \Omega$, this geometric point is defined on the affine scheme in question by $r_{yx}^{-1}(\p)$ and the same field extension). Then there is a canonical isomorphism of profinite groups
\begin{align*}
\underset{{y\in X}}{\mathrm{colim}}\text{ }\pi_1^{et}(\mathrm{Spec}(\OO_{X, y}), \overline{x}_y)\overset{\sim}{\to}\pi_1^{et}(X, \overline{x}).
\end{align*}
\end{theorem}
In more abstract terms, we can say $\overline{x}$ defines a codatum of profinite groups on $X$ and its <<global cosections>> compute the actual fundamental group. One can generalize this to not necessarily irreducible spaces $X$ considering either: 1) a system of geometric points with fixed \textit{Tannaka paths} between the fiber functors (replicating \cite{van kampen}); 2) a system of geometric points and computations up to inner automorphism; 3) fundamental groupoids instead of fundamental groups. Studying coproducts of schematic spaces along flat immersions (see Remark \ref{remark flat immersions}) also allows us to prove (as a corollary of the general <<internal>> version of the Theorem) that the fundamental group or groupoid is a costack for the Grothendieck topology of flat immersions, recovering a more classical formulation. This last result is known for the \'etale fundamental group of schemes and the \'etale topology (\cite{costack}), hence we expect it to hold for an analogue of the pro-\'etale topology for schematic spaces.
\medskip

Recall now Example \ref{example:schematic-not-scheme}, the most elementary case of a schematic space that is not a scheme: $X=\{x, y, z\}$, $x, y<z$, $\OO_{X, x}=A[x]$, $\OO_{X, y}=A[x]$, $\OO_{X, z}=K(x)$. By the Seifert-Van Kampen if $\eta$ is the generic point of $X$, i.e. $\eta\colon (\star, K)\to X$ given by $K(x)\to K$, then one has
\begin{align*}
\pi_1^{et}(X, \eta)\simeq \pi_1^{et}(\mathbb{A}^1_A, \eta_{A[x]})\ast_{\pi_1^{et}(\Spec(K(x)), \eta)}\pi_1^{et}(\mathbb{A}^1_A, \eta_{A[x]})
\end{align*}
where $\mathbb{A}_A^1:=\Spec(A[x])$, $\ast$ denotes the profinite amalgamation and $\eta_{A[x]}$ is the generic point of $\mathbb{A}_A^1$. If we take $A=\mathbb{Q}$, it is known that $ \pi_1^{et}(\mathbb{A}^1_{\mathbb{Q}}, \eta_{\mathbb{Q}[x]})\simeq \pi_1^{et}(\Spec(\mathbb{Q}))=\mathrm{Gal}(\mathbb{Q}_{\mathrm{sep}}/\mathbb{Q})$ (see \cite[6.23]{Lenstra}), so
\begin{align*}
\pi_1^{et}(X, \eta)\simeq \mathrm{Gal}(\mathbb{Q}_{\mathrm{sep}}/\mathbb{Q})\ast_{\mathrm{Gal}(\mathbb{Q}(x)_{\mathrm{sep}}/\mathbb{Q}(x))}\mathrm{Gal}(\mathbb{Q}_{\mathrm{sep}}/\mathbb{Q}).
\end{align*}
Notice that schematic finite spaces have more structure than simplicial schemes so it may be possible to compute $\pi_1^{et}(X, \eta)$ by different means, for instance  using different finite models (within the same qc-equivalence class) or relating them with other spaces. We are currently working on general ways to effectively use this additional structure: a finite schematic space is essentially a locally ringed space \textit{and} an open covering, but unlike in the case of \v{C}ech nerves, we preserve the combinatorial information of the covering; similarly, we also expect to produce <<hyperschematic covers>> of spaces as a richer version of hypercoverings of a space. This is a long-term goal that we hope will lead to a generalized \'etale homotopy theory that might shed some light into the components of this presentation of $\pi_1^{et}(X, \eta)$.

A modification of this example would be the gluing of projective lines. The finite model of the projective line $\mathbb{P}_k^1$ for a field $k$ that is universal among all ringed spaces  (see \cite{universal}) is, up to isomorphism,  the space $\{a, b, c\}$ with $a, b\leq c$ and $\OO_{\mathbb{P}_k^1, a}=[k[x, y]_x]_0$, $\OO_{\mathbb{P}_k^1, b}=[k[x, y]_y]_0$ and $\OO_{\mathbb{P}_k^1, c}=[k[x, y]_{xy}]_0$, where $[R]_0$ denotes the component of degree $0$ of a graded ring $R$. Our <<double projective line>> space (that is schematic and is not a scheme) is the space $\mathbb{PP}_k^1=\{a_1, b_1, c_1, a_2, b_2, c_2, d\}$ with $a_i, b_i\leq c_i$ (for $i=1, 2$) and $d$ a maximal point, such that $\{a_i, b_i, c_i\}\subseteq Y$ is the projective line described above and $\OO_{Y, d}=k(x, y)$. The generic point $\eta$ is given by the zero ideal of $k(x, y)$. An application of Van Kampen's Theorem and \cite[6.22]{Lenstra} tell us that 
\begin{align*}
\pi_1^{et}(\mathbb{PP}_k^1, \eta)\simeq \mathrm{Gal}({k}_{\mathrm{sep}}/{k})\ast_{\mathrm{Gal}({k}(x, y)_{\mathrm{sep}}/{k(x, y)})}\mathrm{Gal}({k}_{\mathrm{sep}}/{k}).
\end{align*}

\medskip
Besides gluing at the generic point, we can also glue spaces at any local ring. Of course, the fundamental groups that appear will, in general, no longer be the amalgamation of Galois groups of fields. Recall that, for any scheme $S$ and a point $s\in S$, the morphism 
\begin{align*}
\Spec(\OO_{S, s})\to S
\end{align*}
is a flat monomorphism, but not an open immersion (it is a limit of open immersions). The theory of schematic spaces allows us to consider gluings along these local rings. An interesting case happens when $s\in S$ is a singular point, i.e. $\OO_{S, s}$ is not a regular ring. We can consider, for example, the affine nodal curve given by the ring $R=k[x, y]/(y^2-x(x^2+1))$ and glue two copies of it along the singular <<point>> given by the maximal ideal $(x, y)$. This would be the space $N=\{a, b, c\}$ with $a, b<c$ and $\OO_{N, a}=\OO_{N, b}=R, \OO_{N, c}=R_{(x, y)}$. If we put $k=\mathbb C$ (for simplicity), the fundamental group of $\Spec(R)$ can be computed via analytic means and is isomorphic to $\widehat{\mathbb{Z}}\simeq \prod_{p\text{ prime}}\mathbb{Z}_p$. In particular, at the generic point $\eta$, we have
\begin{align*}
\pi_1^{et}(N, \eta)\simeq \left(\prod_{p\text{ prime}}\mathbb{Z}_p\right)\ast_{\pi_1^{et}(\Spec(R_{(x, y)}), \eta)}\left(\prod_{p\text{ prime}}\mathbb{Z}_p\right).
\end{align*}
In a similar setting, we may modify this example to consider affine surfaces (or open affines of a general surface) with global sections ring $A$ and $p\equiv \mathfrak{m}\in\Spec_{\mathrm{max}}(A)$ a singular point. The fundamental group $\pi_1^{et}(\Spec(A_{\mathfrak{m}}), \eta)$ is now related to the \textit{local \'etale fundamental group} of the singular point, which contains information about the singularity type and is defined as
\begin{align*}
\pi_1^{loc, et}(\Spec(A), p)=\pi_1^{et}(\Spec(A_{\mathfrak{m}}^{sh})-\{p\}, \eta),
\end{align*}
where the superscript $sh$ denotes the strict Henselization and $\eta$ is a chosen generic geometric point. Note that this makes sense because $\Spec(A_{\mathfrak{m}}^{sh})-\{p\}$ is connected. There is a natural morphism $$\pi_1^{loc, et}(\Spec(A), p)\to  \pi_1^{et}(\Spec(A_{\mathfrak{m}}), \eta)$$ given by the inclusion and the strict Henselization map. This situation generalizes to schemes $S$ of higher dimension and singular subschemes $S_{\mathrm{sing}}\subseteq S$ of codimension $\geq 2$.

Even more generally, given any scheme $S$ with singular points, we can consider finite models with respect to certain <<almost finite>> open coverings and isolate the local rings of singular points \textit{as actual open subsets} in order to later study their fundamental groups through manipulations in the schematic category (as briefly hinted at above). In turn, these are related to the local \'etale fundamental groups. Note as well that normalizations of schematic finite spaces are computed pointwise. Furthermore, introducing a theory of \'etale finite models (of algebraic spaces)---where everything should work in an analogous way---we could even consider gluings along strict Henselizations of local rings (i.e. the local rings in the \'etale topos, instead of the Zariski topos).


\begin{thebibliography}{99} 

\bibitem{scholze} B. Bhatt, P. Scholze, \textit{The pro-\'etale topology for schemes}, Ast\'erisque No. 369 (2015), 99-201.

\bibitem{edgepaths} J. A. Barmak, E. G. Minian, \textit{G-colorings of posets, coverings and presentations of the fundamental group}, arXiv:1212.6442v2 [math.AT].

\bibitem{bourbaki} N. Bourbaki, \textit{Alg\`ebre commutative: Chapitres 1 \`a 4}, Masson, Paris, 1985.

\bibitem{preprint cohaces} V. Carmona S\'anchez, C. Maestro P\'erez, F. Sancho de Salas, J.F. Torres Sancho, \textit{Homology and Cohomology of Finite Spaces}, J. Pure Appl. Algebra 224 (2020), no. 4, 106200, 38pp.

\bibitem{Gillam} W.D. Gillam, \textit{Localization of ringed spaces}, Advances in Pure Mathematics, Vol. 1, No. 5 (2011), 250-263. doi: 10.4236/apm.2011.15045.

\bibitem{EGAIV} A. Grothendieck, \textit{\'El\'ements de g\'eom\'etrie alg\'ebrique. IV. \'Etude locale des sch\'emas et des morphismes de sch\'emas IV}, Inst. Hautes \'Etudes Sci. Publ. Math. No. 32 1967.

\bibitem{SGAI} A. Grothendieck, \textit{Rev\^etements \'etales et groupe fondamental (SGA 1)}, Documents Math\'ematiques (Paris) [Mathematical Documents (Paris)], 3. Soci\'et\'e Math\'ematique de France, Paris, 2003


\bibitem{Hakim} M. Hakim, \textit{Topos annel\'es et sch\'emas relatifs}, Ergeb. Math. Grenz. Band 64. Springer-Verlag 1972.

\bibitem{lidia} L. A. H{\"u}gel, F. Marks, J. {\v S}{\v t}ov{\'i}{\v c}ek, R. Takahashi, J. Vit\'oria, \textit{Flat ring epimorphisms and universal localizations of commutative rings}, The Quarterly Journal of Mathematics, Vol. 71 (2020), Issue 4, 1489-1520, doi: 10.1093/qmath/haaa041.

\bibitem{adic tame} K. H\"ubner, \textit{The adic tame site}, arXiv: arXiv:1801.04776v5 [math.AG].

\bibitem{krauser} H. Krause, \textit{Thick subcategories of modules over commutative Noetherian rings. With an appendix by S. Iyengar}, Math. Ann. 340 (2008), 733-747.

\bibitem{lazard} D. Lazard, \textit{S{\'e}minaire Samuel. Alg{\`e}bre commutative}, Tome 2 (1967-1968), Expos{\'e} no. 4, 12 pp.

\bibitem{Lenstra} H. W. Lenstra, \textit{Galois theory for schemes}, Course  notes, Universiteit  Leiden  Mathematics  Department,  \texttt{
https://websites.math.leidenuniv.nl/algebra/GSchemes.pdf},  Electronic  third edition: 2008.

\bibitem{McCord} M.C. McCord, \textit{Singular homology groups and homotopy groups of finite topological spaces}, Duke Math. J. 33 (1966), 465-474.

\bibitem{costack} I. Pirashvli, \textit{The \'etale fundamental groupoid as a $2$-terminal costack}, Kyoto J. Math. 60 (2020), no. 1, 379-403.

\bibitem{Raynaud} M. Raynaud, \textit{Un crit\`ere de effectivit\'e de descente}. In: S\'eminaire Samuel, Alg\`ebre Commutative, vol. 2, pp. 1-22 (1967-1968).

\bibitem{universal} J. S\'anchez Gonz\'alez, F. Sancho de Salas, \textit{Universal Ringed Spaces}, arXiv:2101.02126v1[math.AG]

\bibitem{Fernando schemes} F. Sancho de Salas, \textit{Finite Spaces and Schemes}, J. Geom. Phys. 122 (2017), 3-27.

\bibitem{Fernando homotopy} F. Sancho de Salas, \textit{Homotopy of finite ringed spaces}, J. Homotopy Relat. Struct. 13 (2018), no. 3, 481-501. 

\bibitem{fernandoypedro} F. Sancho de Salas, P. Sancho de Salas, \textit{Affine Ringed Spaces and Serre's Criterion}, Rocky Mountain J. Math. 47 (2017), no. 6, 2051-2081.

\bibitem{Fernando y JF} F. Sancho de Salas, J.F. Torres Sancho, \textit{Derived category of Finite Spaces and Grothendieck Duality},  Mediterr. J. Math. 17, 80 (2020), doi: 10.1007/s00009-020-01509-3.

\bibitem{van kampen} J. Stix, \textit{A general Seifert-Van Kampen theorem for algebraic fundamental groups}, Publ. Res. Inst. Math. Sci. 42 (2006), no. 3, 763-786.


\bibitem{Szamuely} T. Szamuely, \textit{Galois groups and fundamental groups}, Cambridge Studies in Advanced Mathematics, 117. Cambridge University Press, Cambridge, 2009.  

\bibitem{prufer spaces} M. Temkin, I. Tyomkin, \textit{Pr\"ufer algebraic spaces.} Math. Z. 285 (2017), no. 3-4, 1283-1318.



\end{thebibliography}
\end{document}